\documentclass[12pt]{amsart}
\usepackage{graphicx,amssymb}
\usepackage{mathrsfs}
\usepackage{amssymb,amsmath,amsthm,color}
\usepackage{graphicx,mcite}
\usepackage{hyperref}
\usepackage{url}
\usepackage{setspace}

\theoremstyle{definition}

\theoremstyle{remark}

\numberwithin{equation}{section}

\newcommand{\set}[1]{\left\{#1\right\}}
\newcommand{\suchthat}{\ensuremath{\, \vert \,}}
\newcommand{\R}{\mathbb R}
\newcommand{\C}{\mathbb C}

\newcommand{\LL}{ {\mathcal L} }

\def\TagOnRight

\def\R {\mathbb{R}}

\newcommand{\be}{\begin{equation}}
\newcommand{\ee}{\end{equation}}
\newcommand{\bea}{\begin{eqnarray}}
\newcommand{\eea}{\end{eqnarray}}
\newcommand{\Bea}{\begin{eqnarray*}}
\newcommand{\Eea}{\end{eqnarray*}}
\newcommand{\bt}{\begin{Theorem}}
\newcommand{\et}{\end{Theorem}}

\newcommand{\bpr}{\begin{Proposition}}
\newcommand{\epr}{\end{Proposition}}

\newcommand{\bl}{\begin{Lemma}}
\newcommand{\el}{\end{Lemma}}
\newcommand{\bi}{\begin{itemize}}
\newcommand{\ei}{\end{itemize}}

\newtheorem{Definition}{Definition}[section]
\newtheorem{Theorem}[Definition]{Theorem}
\newtheorem{Lemma}[Definition]{Lemma}
\newtheorem{Proposition}[Definition]{Proposition}

\newtheorem{Remark}[Definition]{Remark}

\begin{document}
\baselineskip16pt

\title[]{ Nonlinear Schr\"{o}dinger equation for the twisted Laplacian  }
\author{P. K. Ratnakumar and Vijay Kumar Sohani}%
\address {Harish-Chandra Research Institute, Allahabad-211019
 India}
 \email { ratnapk@hri.res.in,sohani@hri.res.in}
\subjclass[2010]{Primary 42B37, Secondary 35G20, 35G25}
\keywords{ Twisted Laplacian(special hermite operator), Nonlinear Schr\"{o}dinger equation, 
Strichartz estimates, well posedness }%

\begin{abstract} We establish the local well posedness of 
solution to the nonlinear Schr\"{o}dinger equation associated 
to the twisted Laplacian on $\C^n$ in certain first order Sobolev space. 
Our approach is based on Strichartz type estimates, and is valid  for a  general class of 
nonlinearities including power type. The case $n=1$ represents the magnetic 
Schr\"{o}dinger equation in the plane with magnetic potential $A(z)=iz,~z\in\C$.
\end{abstract}

\maketitle

\section{Introduction}

The free Schr\"{o}dinger equation on $\R^n$ is the PDE
$$ i \partial_t \psi(x,t) + \Delta \psi(x,t)=0, ~~~~~~~~~~~~x \in \R^n, ~t \in
\R$$ which  gives the quantum mechanical description  of the evolution of a free particle in $\R^n$.
If $\psi$ is the solution of the Schr\"{o}dinger equation, then
$|\psi(x,t)|^2$ is interpreted as the probability density for finding the position of the particle
in $\R^n$ at a given time $t$. Let us consider the initial value problem
 \bea
  i \partial_t u(x,t) + \Delta u(x,t)&=&0, ~~~~~~~~~~~~x \in \R^n, ~t \in
\R \\
u(x,0)&=& f(x).
\eea
For $f \in L^2(\R^n)$, the solution is given by the Fourier transform:
$$u(x,t)= \int_{\R^n} e^{-it |\xi|^2} \, \hat{f}(\xi) \, e^{ix \xi} \,d \xi. $$ This may be written as  
$$u(x,t)=e^{it\Delta}f(x)$$
interpreting the Fourier inversion formula as the spectral decomposition in terms of the eigenfunctions of the Laplacian, see \cite{S},\cite{S1}.
 
More generally for any self adjoint differential operator $L$ on $\R^n$, having the spectral representation $ L= \int_E \lambda\, dP_\lambda,$
we can associate the  Schr\"{o}dinger propagator  $ \{e^{-itL} : t \in \R \}$ given by
\bea e^{-itL}f= \int_E e^{-it\lambda} dP_\lambda(f)\eea for $f
\in L^2(\R^n)$. Here $dP_\lambda$ denote the spectral
projection for $L$, i.e., a projection valued measure supported on
the spectrum $E $ of $L$, see \cite{RS}.

In this case, the  function $u(x,t) = e^{-it L}f(x)$  solves the initial value problem
for the Schr\"{o}dinger equation for the operator $L$:
\bea i \partial_t u(x,t) -L u(x,t)&=& 0, ~x \in \R^n, ~t \in
\R \\
u(x,0) &=& f(x)
\eea with $L$ now representing the corresponding Hamiltonian of the quantum mechanical system. 

The significance of this view point is that, most Hamiltonians of interest, namely  the perturbation of the Laplacian with a potential $V$ 
(of the form $L =- \Delta +V$) or the magnetic Laplacian corresponding to the magnetic potential $(A_1(x),..., A_n(x))$  (of the form
$L= \sum_{j=1}^n \left(i \partial_{x_j} + A_j(x) \right)^2)$ on $\R^n$, can be analysed with our approach, in terms of  the spectral theory 
of the Hamiltonian. 

\hskip.1in
In this paper, we concentrate on Schr\"{o}dinger equation for an interesting magnetic Laplacian, namely the twisted Laplacian on $\C^n$; 
also known as the special Hermite operator. The twisted Laplacian ${\mathcal L}$ on $\C^n$ is given by
\Bea {\mathcal L}=\frac{1}{2}\,
\sum_{j=1}^n\left(Z_j\overline{Z}_j+\overline{Z}_jZ_j\right)\Eea where 
$Z_j= \frac{\partial}{\partial z_j} + \frac{1}{2}\bar{z}_j,~
\overline{Z}_j= -\frac{\partial}{\partial \bar{z}_j} + \frac{1}{2}
z_j, ~ j=1,2,\ldots,n.$ Here $\frac{\partial}{\partial z_j}$ and
$\frac{\partial}{\partial \bar{z}_j}$  denote the complex derivatives 
$\frac{\partial}{\partial_{x_j}} \mp i \frac{\partial}{\partial_{y_j}}$ 
respectively. The  operator $\LL$ may be viewed as the complex analogue 
of the quantum harmonic oscillator Hamiltonian $H= -\Delta+|x|^2$ on
$\R^n$, which has the representation $$H=
\frac{1}{2}\sum_{j=1}^n\left(A_jA_j^*+A_j^*A_j\right)$$ in terms of the
creation operators $A_j=- \frac{d}{dx_j}+x_j$ and the annihilation
operators $A_j^*= \frac{d}{dx_j}+x_j,~ j=1,2,\ldots, n$.  The operator 
$\LL$  was introduced by R. S. Strichartz \cite{S1}, and called the  
special Hermite operator and it looks quite similar to the Hermite operator on $\C^n$. In fact
$$ {\mathcal L}=-\Delta + \frac{1}{4}|z|^2 - i
\sum_1^n\left(x_j \frac{\partial}{\partial y_j}-y_j
\frac{\partial}{\partial x_j}\right). $$

This may be re written as
$$\LL =\overset{n}{\underset{j=1}{\sum}} \left[ \left(i \partial_{x_j} 
+\frac{y_j}{2} \right)^2 +\left(i \partial_{y_j} -\frac{x_j}{2} \right)^2 \right]$$
which is of the form $\overset{2n}{\underset{j=1}{\sum}} 
\left[(i \partial_{w_j} -A_j(w))^2  \right]$ hence represents a 
Schr\"{o}dinger operator  on $\C^n$ for the magnetic vector 
potential $A(z) =iz, z \in \C^n$. 

The Schr\"{o}dinger equation for the magnetic potential with magnetic field  decaying at 
 infinity has been studied by many authors, see for instance Yajima \cite{YA}, where he studies the
propagator for the linear equation. In contrast, the nonlinear equation in our situation corresponds to a 
magnetic equation with a constant magnetic field, which has no decay. For more details on general 
magnetic Schr\"{o}dinger equation corresponding to magnetic field without decay, see  \cite{AHS}.

We consider the initial value problem for the nonlinear 
Schr\"{o}dinger equation for the twisted Laplacian 
$\LL$ : \bea \label{mainpde}i\partial_tu(z,t)\!\!\!\!& -\!\!\!&\LL
u(z,t)=G(z, t,u), ~~~~~~~~~~~~z \in \C^n, ~t \in \R \\
\label{inidata}&& u(z,t_0)= f(z) \eea
where $G$ is a suitable $C^1$ function on $\C^n \times \R \times \C$.
When $G\equiv 0$, and  $f\in L^2(\C^n)$ the solution to this
initial value problem is given by $$u(z,t)= e^{-i(t-t_0)\LL}f(z).$$ 
When $G(z,t,u) = g(z,t)$, the solution is given by the Duhamel formula
\bea \label{1.3}
u(z,t)= e^{-i(t-t_0){\mathcal L}} f(z) -i \int_{t_0}^t e^{-i(t-s){\mathcal
L}}g(z,s)ds. \eea  Thus in the linear case, the solution is 
determined once the functions $f$ and $g$ are known.

For simplicity, we take  $t_0=0$. The basic idea in the  nonlinear analysis is the following heuristic 
reasoning based on the above formula. If the solution $u$ is known, 
then one would expect $u$ to satisfy the above equation with $g(z,s)$ 
replaced by $G(z,s,u(z,s))$:
\bea \label{solution} u(z,t)= e^{-it{\mathcal L}} f(z) -i 
\int_0^t e^{-i(t-s){\mathcal L}}G(z,s,u(z,s))ds. \eea

Indeed one can show that $u$ from a reasonable function space satisfies 
a PDE of the form (\ref{mainpde}), (\ref{inidata}), if and only if  $u$ 
satisfies an integral equation of the form (\ref{solution}),  see Lemma 
\ref{equivalence} for a precise version of this fact.

This reduces the existence theorem for the solution to the nonlinear 
Schr\"{o}dinger equation to a fixed point theorem for
the operator 
\bea \label{operator} {\mathcal H}(u)(z,t)=e^{-it\mathcal{L}}f (z)
-i\displaystyle\int_{0}^t e^{-i(t-s)\mathcal{L}}G(z,s,u(z,s))ds \eea
in a 
 suitable subset of the relevant function space. 

The nonlinearity $G$, that we consider is a $C^1$ function of the form
\bea\label{nlc1}G(z,t,w) = \psi(x,y,t,|w|)\, w, ~ (x,y,t,w)
\in \R^n\times \R^n \times \R \times \C,\eea where $z=x+iy
\in\C^n,t\in\R,w\in\C $ and
 $\psi:  \R^n\times \R^n \times \R \times [0,\infty)\to \C$
is assumed to satisfy the following condition: The function 
$F= \psi, \partial_{x} \psi $, $ \partial_{y}\psi$ and $|w|\partial_4\psi(x,y,t, |w|)$, satisfy the inequality
\bea \label{nlc2} |F(x,y,t, |w|)| \leq C |w|^ {\alpha},
\eea for some constant $C$ and $\alpha \in [0, \frac{2}{n-1})$.

The  class of nonlinearity given by (\ref{nlc1}), (\ref{nlc2}) includes in 
particular, power type nonlinearity of the form $|u|^\alpha u$. 
Moreover, the above class seems  to be the most general form of 
nonlinearity adaptable to the Schr\"{o}dinger equation for the 
twisted Laplacian, for local existence via Kato's method \cite{Ka}.
The main difficulties in this approach is caused by the noncommutativity 
of $\LL$ with $\frac{\partial}{\partial x_j},\frac{\partial}{\partial y_j}$
and the noncompactibility of $\LL$ with the powertype nonlinearty as observed in  \cite{CE}.
However we are able to overcome  this difficulty by introducing the appropriate  set of  differential 
operators $L_j$, $M_j, ~j=1,2,...,n$ and working with a suitable  Sobolev space defined using
 these operators,  see Section \ref{regularity}.

We follow Kato's approach, using Strichartz estimates\cite{S},  as indicated above,
to establish local existence.   The advantage of Kato's method is that, it is useful, even when the conservation laws are not available. 
The main Strichartz estimate  for the Schr\"{o}dinger propagator for the twisted Laplacian (i.e., the special hermite operator )
has already been proved in \cite{Ra}. We also need some more relevant 
estimates like the associated retarded estimates etc., which we prove here (Theorem \ref{strichartzs}).

There is a vast literature available for well posedness results for nonlinear Schr\"{o}dinger equation on $\R^n$.
See for instance the papers by Ginibre and Velo, \cite{GV}, \cite{GV1},  Kato \cite{Ka},
the result of Cazenave and Weisler \cite{CW}, the books by Cazenave \cite{C} and Tao\cite{TT} and the extensive 
references therein. Some of the references that we came across dealing with
magnetic Schr\"{o}dinger equation are \cite{YA}, \cite{AHS} and \cite{CE} as mentioned before. In fact, the stability result 
discussed in \cite{CE}, is actually the stability problem for the nonlinear Schr\"{o}dinger equation for the twisted Laplacian in the plane.  

Our main results in this paper is the well posedness for the nonlinear 
Schr\"{o}dinger equation for the twisted Laplacian, in  the  Sobolev space $\tilde{W}^{1,2}(\C^n)$ 
(see Section \ref{regularity} for the definition) and is given by  the following two theorems:

 \bt \label{thm1} (Local existence)
 Assume that $G$ is as in (\ref{nlc1}), (\ref{nlc2}) and
$u(z,0)=f(z)\in \tilde{W}^{1,2}(\C^n)$. Then there exists a number $T=T(\| u_0\|)$ such that the initial 
value problem (\ref{mainpde}), (\ref{inidata}) has a unique solution 
$u \in C([-T,T]; \tilde{W}^{1,2}(\mathbb{C}^n))$.
\et

Note that the interval [-T,T] given by Theorem \ref{thm1} need not be 
the maximal interval on which the solution exists. Now we discuss the uniqueness and stability in $\tilde{W}^{1,2}(\C^n)$
on the maximal interval denoted by $(T_- , T_+)$.

\begin{Definition} Let $\{f_m\}$ be any sequence such that $f_m\to f$ in $\tilde{W}^{1,2}(\C^n)$  and let $u_m$ be the solution corresponding to the initial value $f_m$.
We call the solution to the initial value problem 
(\ref{mainpde})(\ref{inidata})  stable in $\tilde{W}^{1,2}(\C^n)$
if  $u_m\rightarrow u$ in 
$C(I,\tilde{W}^{1,2})$ for any compact interval $I\subset (T_-,T_+)$.
\end{Definition}  

\bt \label{wellposedness} The local solution established in Theorem \ref{thm1}
extends to a maximal interval $(T_-,T_+)$, where the following blowup alternative holds: either $|T_\pm|=\infty$ 
or $\underset{t \to T_\pm}{\lim} \|u(t,\cdot)\|_{\tilde{W}^{1,2}} = \infty$.
Moreover, the solution is unique and stable in $C((T_-,T_+),\tilde{W}^{1,2})$.  
\et
\begin{Remark} In the definition of stability we can consider only the compact intervals and 
not the maximal interval. In fact $u_m$ may not be defined on the maximal interval $(T_-,T_+)$ for $u$.
Also by blowup alternative $u$ may not be in $L^{\infty}((T_-,T_+),\tilde{W}^{1,2})$. 
\end{Remark}

 The paper is organised as follows. In Section \ref{spectral} we discuss the 
spectral theory of the twisted Laplacian $\LL$ and introduce 
the Schr\"{o}dinger propagator $e^{it\LL}$, and prove the relevant Strichartz 
estimates. In Section \ref{regularity} we introduce certain first order Sobolev spaces
associated to the twisted Laplacian, prove an embedding result 
in $L^p $ spaces and other auxiliary estimates. The proofs of the main theorems are
presented in Section \ref{local existence}.\\

\noindent
{\bf Acknowledgements:} 
This work is part of the Ph. D. thesis of the second author. He wishes to thank the Harish-Chandra Research institute,  
the Dept. of Atomic Energy, Govt. of India, for providing excellent research facility. 

\section{Spectral theory of the twisted Laplacian}\label{spectral}
The twisted Laplacian is closely related to the sub Laplacian on the 
Heisenberg group, hence the spectral theory of this operator is 
closely connected with the representation theory of the Heisenberg group.
Here we give a brief review  of the spectral theory of the twisted Laplacian  $\LL$. 
The references for the materials discussed in this section are the following books:
Folland \cite{F}, and Thangavelu \cite{T}, \cite{T1}.

The eigenfunctions of the operator ${\mathcal L}$ are called the
special Hermite functions, which are defined in terms of the
Fourier-Wigner transform. For a pair of functions $f,g \in
L^2(\R^n)$, the Fourier-Wigner transform is defined to be
$$V(f,g)(z)= (2 \pi)^{-\frac{n}{2}} \int_{\R^n} e^{ix \cdot \xi}
f\left(\xi+\frac{y}{2}\right) \overline{g}\left(\xi-\frac{y}{2}\right) d \xi,$$ where $z=x+iy \in
\mathbb{C}^n.$  For any two multi-indices $\mu, \nu$ the special
Hermite functions $\Phi_{\mu \, \nu}$ are given by
$$\Phi_{\mu \, \nu}(z)=V(h_\mu, h_\nu)(z)$$ where $h_\mu$ and $h_\nu$ are Hermite
functions on $\R^n$. Recall that for each nonnegative integer $k$,
the one dimensional Hermite functions $h_k$ are defined by $$
h_k(x)= \frac{(-1)^k}{\sqrt{2^k k!
\sqrt{\pi}}}\left(\frac{d^k}{dx^k}e^{-x^2}\right)
e^{\frac{x^2}{2}}.$$ Now for each multi index $\nu=(
\nu_1,\cdots,\nu_n)$, the n-dimensional Hermite functions are
defined by the tensor product : $$h_\nu(x)= \prod_{i=1}^n
h_{\nu_i}(x_i), ~~~~~~~x=(x_1,\cdots,x_n).$$

A direct computation using the relations  \Bea
\left(-\frac{d}{dx}+x \right)h_k(x)&=& (2k+2)^{\frac{1}{2}}h_{k+1}(x),\\
\left(~~\frac{d}{dx}+x\right)h_k(x)&=& (2k)^{\frac{1}{2}}
h_{k-1}(x)\Eea satisfied by the Hermite functions $h_k$ show that
${\mathcal L}\Phi_{\mu \, \nu}= (2|\nu|+n)\Phi_{\mu \, \nu}$. Hence
$\Phi_{\mu\, \nu}$ are eigenfunctions of ${\mathcal L}$ with
eigenvalue $2|\nu|+n$ and they also form a complete orthonormal
system in $L^2(\C^n)$. Thus every $f \in L^2 (\C^n)$ has the
expansion \bea f= \sum_{\mu, \, \nu} \langle f, \Phi_{\mu \,
\nu}\rangle \Phi_{\mu \, \nu} \eea in terms of the eigenfunctions of
${\mathcal L}$. The above expansion may be written as
\bea\label{2.2} f= \sum_{k=0}^\infty
P_kf\eea where \bea P_kf=\sum_{\mu, |\nu|=k} \langle f, \Phi_{\mu,
\nu}\rangle \Phi_{\mu \nu} \eea is the spectral projection
corresponding to the eigenvalue $2k+n$. Now for any $f\in L^2(\C^n)$
such that ${\mathcal L}f \in L^2(\C^n),$ by self adjointness of
${\mathcal L}$, we have $P_k(\LL f)= (2k+n) P_k f$. It follows that
for $f\in L^2(\C ^n)$ with  $\LL f \in L^2(\C^n)$
 \bea \label{2.4} \LL f=
\sum_{k=0}^\infty (2k+n)P_kf.\eea Thus, we can
define $e^{-it\LL}$ as \bea e^{-it\LL}f = \sum_{k=0}^\infty
e^{-it(2k+n)} P_k f.\eea

Note that $P_kf$ has the compact representation
$$P_kf(z)=(2\pi)^{-n} (f\times\varphi_k)(z) $$ in terms of the Laguerre function
 $\varphi_k(z)=L_k^{n-1}(\frac{1}{2}|z|^2)e^{-\frac{1}{4}|z|^2}$, see \cite{T}.
 Hence formally we can express $e^{-it\LL} $ as a twisted convolution operator: 
 $$e^{-it\LL}f=f\times K_{it}$$ for $f\in \mathcal{S}(\C^n)$
 where  $K_{it}(z)= \frac{(4\pi i)^{-n}}{(\sin t)^n} e^{\frac{i(\cot t)|z|^2}{4}}$.
Crucial to the local existence proof is a Strichartz type estimate for the
Schr\"{o}dinger propagator for the twisted Laplacian.
We start with the following definition
\begin{Definition}\label{addmi}
Let $n\geq 1$. We say that a pair $(q,p)$ is {\it admissible} if 
 \[ 2<q< \infty \,\, {\rm and }\,\, {1\over q}\geq
n\left({1\over 2}-{1\over
p}\right)\geq 0.\]
\end{Definition}
\begin{Remark} The admissibility condition on $(q,p)$ implies that $2\leq p <\frac{2n}{n-1}$.
\end{Remark}
The  Strichartz type estimate for the Schr\"{o}dinger propagator 
for the twisted Laplacian has been established in \cite{Ra} and we state here a variant of 
the lemma on the convolution on the circle 
proved there. The proof follows exactly as in \cite{Ra}.

\begin{Lemma}\label{convcirc}:
Let $K \in \mbox{weak}~ L^\rho([-a,a])$ for some $\rho>1$ and let $T$ be the  operator  given by
$$Tf(t)=\int_{[-a,a]} K(t-s)f(s)ds.$$
Then the following inequality holds
\Bea
\Vert Tf \Vert_q &\leq& C_K \Vert f\Vert_{q'}  ~ ~ \mbox{ for} ~q=2\rho 
\Eea
with  $C_K= C[K]_\rho$, where $[K]_\rho$ denotes the weak $L^{\rho}([-a,a])$ norm of $K$. 
\end{Lemma}
We use the compact notation $L_{[-T,T]}^{p,q}$ or simply $L^{p,q}$ for 
$L^q([-T,T],L^p(\C^n))$ and $L^{2,\infty}$ for the mixed $L^p$ space $L^{\infty}([-T,T],L^2(\C^n))$.
The main Strichartz type estimate we require is compiled in the following
\begin{Theorem}\label{strichartzs} Let $f \in L^2(\C^n)$ and $g \in L^{q'} 
([-a,a] ;\, L^{p'} (\C^n) )$ where $(q,p)$ denote an admissible pair with $ q'$ and $p'$ denoting the corresponding
 conjugate indices. 
Then for $0<a<\infty$, the following estimates hold over $[-a,a]\times \C^n$:
\bea\label{mainstrest} \| e^{-it\LL} f \|_{L^{p,q} } &\leq&
C \|f \|_2 \\
\label{stauxest} \left\| \int_0^t e^{-i(t-s)\LL} g(z,s) ds 
\right\|_{L^{p,q}} &\leq& C \|g \|_{L^{p',q'} } \\
 \label{retstrest} \left\| \int_0^t e^{-i(t-s)\LL} g(z,s) ds 
\right\|_{L^{2,\infty}} &\leq& C \,  \|g \|_{L^{p',q'} }
 \eea
with a constant $C$ is independent of $f$ and $g$.
\end{Theorem}

\begin{proof}
The proof of the estimate (\ref{mainstrest}) given in  \cite{Ra} for $a= \pi$, works for any $a<\infty$, 
and relies on the following dispersive estimate for the complex semigroup $ e^{-(r+it)\LL}$  (see \cite{Ra}):
\bea \label{pp'est}\| e^{-(r+it)\LL}f(z) \|_p 
\leq 2 
|\sin t|^{-2n (\frac{1}{2}-\frac{1}{p})} \,\| f\|_{p'}
\eea
valid for all  $f \in L^p(\C^n),~ 2\leq p \leq \infty,~r>0$ combined with a limiting argument as $r\to 0$. 

Now we will give a direct proof of (\ref{pp'est}) for $r=0$, essentially using the regularization argument used in \cite{Ra}, 
to deduce (\ref{stauxest}) and (\ref{pp'est}). The regularisation technique was first introduced in \cite{NR}, \cite{NR1}, see also \cite{Ra1}.
The results in these cases do not follow by the general Strichartz estimates established by Keel and Tao in \cite{KT} by lack of 
dispersive estimate for the kernel for the Schr\"{o}dinger propagator in these cases.  

For $f\in L^1\cap L^2(\C^n),$ we have the isometry $\| e^{-it\LL}f \|_2= \|f\|_2$. Hence, using the series expansion for $ e^{-(r+it)\LL}f$ and appealing to the dominated convergence theorem for the sum, we can see that  for any sequence $r_m\rightarrow 0$, $e^{-(r_m +it)\LL}f\rightarrow e^{-it\LL}f$ in 
$L^2(\C^n)$, it follows that $e^{-(r+it)\LL}f(z) \rightarrow e^{-it\LL}f(z)$ for almost all $z\in \C^n$ as $r\rightarrow 0$. This gives the inequality $\|e^{-it\LL}f(z)\|_{L^{\infty}(\C^n)} 
\leq \frac{2}{|\sin t|^n}\|f\|_{L^1}$ from the corresponding inequality for the complex semigroup $e^{-(r+it)\LL}$. Now the inequality (\ref{pp'est}) for $r=0$ follows by interpolating this  with the $L^2$ isometry. 

The estimate (\ref{stauxest}) can be deduced from the inequality 
(\ref{pp'est}) with $f(z)$ replaced by $g(z,t)$ for each $t$, 
and Lemma \ref{convcirc}. In fact using Minkowski's inequality for integrals
and (\ref{pp'est}) with $r=0$, we get 
\Bea
 \left\| \int_0^t e^{-i(t-s)\LL} g(z,s) ds 
\right\|_{L^p} &\leq& \int _0^a\frac{2}{|\sin (t -s)|^{2n(\frac{1}{2}- \frac{1}{p}) }} \|g (\cdot,s) \|_{L^{p'} } ds. 
\Eea
Also for $p< \frac{2n}{n-1}$, we have $|\sin t|^{-2n(\frac{1}{2}- \frac{1}{p}) } \in$ 
weak $L^\rho[-a,a]$, for any $\rho \leq \frac{1}{2n(\frac{1}{2}- \frac{1}{p})}$. 
Hence by Lemma \ref{convcirc} we see that 
\Bea
 \left\| \int_0^t e^{-i(t-s)\LL} g(z,s) ds 
\right\|_{L^{p,q}} &\leq& C \|g \|_{L^{p',q'} }
\Eea for any $q= 2\rho>2$. Since  $2\rho \leq \frac{1}{n(\frac{1}{2}- \frac{1}{p})},$
the above inequality is valid for any admissible pair, hence (\ref{stauxest}) 

To prove (\ref{retstrest}) take  $h \in {\mathcal S}(\C^n)$. 
By H\"{o}lder's inequality and estimate (\ref{mainstrest}) 
\Bea
\left| \int_{\mathbb {C}^n } \int_0^t e^{i(t-s) \LL} g(z,s) ds \, 
\overline{h(z)} dz \right| &\leq & \int_{-a} ^a  \int_{\mathbb {C}^n } 
\left| g(z,s) \, \overline{ e^{-i(t-s)\LL} \,h(z) } \right|dz ds \\
&\leq&  \|g \|_{L^{p',q'}}   \, \left \|    e^{-i(t-s)\LL}  \,h(z) \right \|_{L^{p,q}}\\
&\leq &  C \|g \|_{L^{p',q'}}   \, \| h(z)\|_2.
\Eea
Taking supremum over all $h$ with $\|h\|_2=1$ we get the required estimate.
\end{proof}

\section{Some auxiliary regularity estimates}\label{regularity}

In this section we establish some auxiliary estimates 
used in the existence proof. The regularity of the solution 
is obtained through certain first order Sobolev space 
$\tilde{W}^{1,p}$ naturally associated to the one parameter 
group $\{ e^{it \LL}: t \in \R\}$, which we now introduce.
Let   $L_j$ and $M_j$ be the  differential operators given 
by $$L_j=\left(\frac{\partial}{\partial
  x_j}+i\frac{y_j}{2}\right), ~~~\mbox{and}~~~ M_j=
\left(\frac{\partial}{\partial
  y_j}-i\frac{x_j}{2}\right), ~~~j=1,2,...,n.$$

\begin{Definition} Let  $m$ be a nonnegative 
integer and $1\leq p< \infty$. We say $f \in \tilde{W}^{m,p}(\C^n)$ if
 $S^\alpha f  \in L^p(\C^n)$ for $|\alpha| \leq m$ where $S^\alpha 
= \prod_{i=1}^{2n} S_i^{\alpha_i}$ with $S_i= L_i$ for $1\leq i \leq n$,
$S_i=M_i$ for $n+1 \leq i\leq 2n$, and $|\alpha| = \alpha_1+\cdots+ \alpha_{2n}$.
 $ \tilde{W}^{m,p}(\C^n)$ is a Banach space with norm given by
$$\Vert f\Vert_{\tilde{W}^{m,p}}=\max \{ \Vert S^\alpha f \Vert_{L^p}:~ 
|\alpha| \leq m \}.$$ 
In particular,  $f \in \tilde{W}^{1,p}(\C^n)$ iff $f ,~ L_jf$ and 
$M_jf  $ are in $ L^p(\C^n)$ for $1\leq j\leq n$ and the norm in
$ \tilde{W}^{1,p}(\C^n)$ is given by
$$\Vert f\Vert_{\tilde{W}^{1,p}}=\max \{\Vert f\Vert_{L^p}, 
\Vert L_jf\Vert_{L^p}, \Vert M_jf\Vert_{L^p},~ j=1,2,\ldots ,n   \}.$$
\end{Definition}

\begin{Remark}
The differential operators $L_j$ and $M_j$ are the natural ones 
adaptable to the power type nonlinearity $G(u)=|u|^\alpha u$ and the
generality that we consider here, see Lemma 
\ref{commutativity} and Proposition \ref{Gest}. The natural choice, 
namely the
standard Sobolev space  $W_\LL^{1,p}(\C^n)$ defined 
using the twisted Laplacian  $\LL$ (see \cite{T2}),
is not  suitable for treating such nonlinearities.
\end{Remark}

\begin{Remark} An interesting relation between the Sobolev space  
$\tilde{W}^{1,p}(\C^n)$ and the ordinary $L^p$ Sobolev space
$W^{1,p}(\C^n)$ is the following: If $u \in \tilde{W}^{m,p}(\C^n)$, 
then $|u| \in W^{m,p}(\C^n)$ for $m=1$. This observation is crucial
in the proof of the  following  embedding theorem for the Sobolev 
space $\tilde{W}^{1,2}$.
\end{Remark}

\begin{Lemma}\label{SobL}[Sobolev Embedding Theorem]\label{embedding} 
We have the continuous inclusion
\begin{equation*}
\begin{array}{lll}
 \tilde{W}^{1,2}\hookrightarrow L^p(\C^n) & \text{ for } 2\leq p \leq
                                                     \frac{2n}{n-1} &\text{ if
                                                     } n \geq 2,  \\
                                                     &\text{ for } 2 \leq
                                                     p < \infty &\text{ if
                                                     }  n=1.

\end{array}
\end{equation*}
\end{Lemma}

\begin{proof} For $u \in {\mathcal S}(\C^n)$, we have
\Bea
2|u| \frac{\partial}{\partial x_j}|u|=\frac{\partial}{\partial x_j}
( \overline{u} u)=
2\Re \left( \overline{u} \,\frac{\partial}{\partial x_j}u\right)= 
2\Re\left(  \overline{u}( \frac{\partial}{\partial x_j}
+\frac{iy_j}{2} )\,u\right).
\Eea
Hence on the set $A=\set{z \in\C^n  \suchthat  u(z)\neq 0 },$ we have
$$ \left| \frac{\partial}{\partial x_j}|u| \right| =  
\left| \Re \left(  \frac{\overline{u}}{|u|} ( \frac{\partial}
{\partial x_j}+\frac{iy_j}{2} )\,u \right) \right|
\leq \left| L_j (u) \right| .$$ 
The  boundary $\partial A = \set{z \in\bar{A}  \suchthat ~ u(z) =0 }$ has Lebesgue measure zero, being a set of lower dimension.
Thus we see that $\left|\frac{\partial}{\partial x_j}|u|\right|\leq|L_ju|$ a.e. on the support of $u$.
Similarly  $ \left|\frac{\partial}{\partial y_j}|u|\right|\leq|M_ju|$ a.e. on the support of $u$.  
It follows that the inequality $\Vert|u|\Vert_{H^1}\leq (2n+1)\,
\Vert u \Vert_{\tilde{W}^{1,2}}$ holds
for all $ u \in {\mathcal S}(\C^n).$
By density of $ {\mathcal S}(\C^n)$ in $H^1$ and $\tilde{W}^{1,2}$, 
the same inequality holds  for all $u \in \tilde{W}^{1,2}$.  Hence the 
result follows from the usual Sobolev embedding theorem on $\R^{2n}$. 
\end{proof}

\begin{Lemma}\label{commutativity}
 The operators $L_j$ and $M_j$ commute with both the operators 
$e^{-it\LL}$ and $\displaystyle\int_{0}^t e^{-i(t-s)\mathcal{L}} ds$,  
for $j=1,2,\ldots,n.$
\end{Lemma}

\begin{proof}
For  $f \in C_c^{\infty}(\C^n)$, we have  $e^{-it\LL}f (z) = f \times K_{it}$, 
where $K_{it}$ denotes the Schr\"{o}dinger kernel: $K_{it}(z)= 
\frac{(4\pi i)^{-n}}{(\sin t)^n} e^{\frac{i\cot t}{4}|z|^2}$.
A direct calculation shows that 
$$\left(\partial_{x_j} + i\frac{y_j}{2} \right)\left[f (z-w)\, 
e^{\frac{i}{2} \Im(z \cdot \bar{w})}
\right]=e^{\frac{i}{2} \Im(z \cdot \bar{w})}\,  \left[ \left(\partial_{w_j}+ 
\frac{i}{2} v_j \right) f \right](z-w)$$
$z=x+iy, w=u+iv$, from which the commutativity of $L_j$ and $e^{-it\LL}$ follows. 
A similar calculation also shows the commutativity of $M_j$ and   
$e^{-it\LL}$ for $j=1,2,\ldots,n$ in $C_c^\infty(\C^n)$.
From the above computation, it follows that the above 
commutativity 
holds in $L^1_{loc}(\C^n)$, treating the locally integrable function 
as a distribution.

The commutativity of $\displaystyle\int_{0}^t e^{-i(t-s)\mathcal{L}}$ with 
$L_j$ and $M_j$ is also a consequence of the above observation.
\end{proof}

For local existence, we need to establish the existence of a 
fixed point for the operator $\mathcal {H}$, defined by
\Bea
{\mathcal H}(u)(z,t)=e^{-it\mathcal{L}}f (z)-i\displaystyle
\int_{0}^t e^{-i(t-s) \, \mathcal{L}}G(z,s,u(z,s)) \, ds. \Eea
In the next two results, we prove some essential estimates 
required for this.

\begin{Lemma}\label{FEST}
 Let $ f \in \tilde{W}^{1,2}(\C^n)$. Then the following estimates hold:

\bea\label{fest1}\Vert e^{-it\mathcal{L}} f \Vert_{L^ \infty\left([-T,T\, ];
\tilde{W}^{1,2}(\C^n)\right)} &=& \Vert f \Vert_{\tilde{W}^{1,2}(\C^n)},\\
\label{fest2}
\Vert e^{-it\mathcal{L}} f \Vert_{L^q\left([-T,T];\tilde{W}^{1,p}(\C^n)\right)} 
&\leq& C\Vert f \Vert_{\tilde{W}^{1,2}(\C^n)}
\eea
for admissible pairs $(q,p)$, with a constant $C$ independent of $f$.

\end{Lemma}
\begin{proof}
Since  both $ L_j$ and $M_j$ commute with the isometry $e^{-it\mathcal{L}}$, 
we have $$\|Se^{-it\mathcal{L}} f \|_{L^2(\C^n)}= \| Sf\|_{L^2(\C^n)}$$ 
for every $t\in \R$ with $S=L_j, ~\mbox{or}~M_j, ~ j = 1,2,...,n$ 
from which Eq. (\ref{fest1}) follows.
Estimate (\ref{fest2}) follows from the Strichartz type estimate 
(\ref{mainstrest}) for $e^{-it\LL}$ using the above commutativity.
\end{proof}

\begin{Proposition}\label{Gest}
Let  $G(z,t,w)$ and $\alpha$ be as in (\ref{nlc1}), (\ref{nlc2}) and $(q,p)$ an admissible pair
with  $p=\alpha +2$, $q> 2$ and $S$ as in Lemma \ref{FEST}. If  
$u \in L^\infty([-T,T];\tilde{W}^{1,2}(\C^n)) \cap {L^q([-T,T];
\tilde{W}^{1,p}(\C^n))}$, then 
$$G(z,t,|u(z,t)|) \in L^{q'}\left([-T,T];\tilde{W}^{1, p'}(\C^n)\right)$$ 
and the following inequalities hold:
\bea \label{Gest1}
&& \left\Vert G(z,t,|u(z,t)|) \right\Vert_{L^{p',q'}}  \\
&\leq& CT^{\frac{q-q\prime}{qq\prime}} 
\|u(z,t)\|_{L^\infty ([-T,T],{\tilde{W}^{1,2}(\C^n))} }^\alpha
\Vert u(z,t)\Vert_{L^q([-T,T]:L^p(\C^n) )}, \nonumber\\
 \label{Gest2}
&& \left\Vert S G(z,t,|u(z,t)|) \right\Vert_{L^{p',q'}} \\ 
\hspace{.5cm}&\leq& CT^{\frac{q-q\prime}{qq\prime}} 
\|u(z,t)\|_{L^\infty ([-T,T],{\tilde{W}^{1,2}(\C^n))} }^\alpha
\Vert  u(z,t)\Vert_{L^q([-T,T]:\tilde{W}^{1,p}(\C^n) )}.  \nonumber
\eea
\end{Proposition}

\begin{proof}
Since $G(z,t,u(z,t)) =\psi(z,t,|u(z,t)|)u(z,t)$, and 
$\frac{q'}{q}+\frac{q-q'}{q}=1$, an application of 
the H\"{o}lder's inequality in the $t$-variable shows that for $q > 2$
\bea\label{est1psi}
\left\Vert G(z,t,|u(z,t)|)\right\Vert_{L^{q'}\left([-T,T];L^{p'}(\C^n)\right)}
&\\ \nonumber
&\hspace{-5cm}\leq
(2T)^{\frac{q-q\prime}{qq\prime}}\Vert\psi(z,t,|u(z,t)|)u(z,t)
\Vert_{L^q\left([-T,T];L^{p\prime}(\C^n)\right)}.
\eea  
By an application of H\"{o}lder's inequality 
in the $z$-variable, using $\frac{p'}{p}+\frac{\alpha p'}{p}=1$,
we see that for a.e. $t\in[-T,T]$
\bea
\| \psi(z,t,|u(z,t)|) \, u(z,t) \|_{L^{p\prime}(\C^n)}
&\leq\| \psi(z,t,|u(z,t)|) \|_{L^{\frac{p}{\alpha}}(\C^n)}\, 
\|u(z,t) \|_{L^p(\C^n)} \nonumber \\
&\hspace{-.6cm} \leq C\| u(z,t)\|^\alpha _{\tilde{W}^{1,2}(\C^n)} 
\| u(z,t)\|_{L^p(\C^n)} \nonumber \\\label{est2psi}
&\hspace{-.2cm}\leq C\|u\|_{L^\infty \left([-T,T],
{\tilde{W}^{1,2}}\right) }^\alpha\Vert u(z,t)\Vert_{L^p(\C^n)},
\eea

\noindent where we used the condition (\ref{nlc2}) on $\psi$ and Lemma 
\ref{embedding} in the second inequality. Now taking 
$L^{q}$ norm with respect to $t$ on both sides, and substituting in
the RHS of inequality (\ref{est1psi})  gives estimate (\ref{Gest1}).

To prove the inequality (\ref{Gest2}), we first observe that
\bea\label{L of upsi} L_j[\psi(x,y,t,|u|) u] &=& \psi(x,y,t,|u|)\, 
L_ju\nonumber  \\
 \!\!\! &+& u \, (\partial_4 \psi)(x,y,t,|u|) \,
\Re\left( \frac{\bar{u}}{|u|}L_ju \right)+u(\partial_{x_j} \psi)(x,y,t,|u|),\\
\label{M of upsi}M_j[\psi(x,y,t,|u|) u] &=&  \psi(x,y,t,|u|)\,
 M_ju  \nonumber \\
\!\!\! &+&\!\!\! u \, (\partial_4 \psi)(x,y,t,|u|) \,
\Re\left( \frac{\bar{u}}{|u|}M_ju \right)+u(\partial_{y_j} \psi)(x,y,t,|u|).
\eea

Thus we see that for $S=L_j$ and $M_j$, $|SG|$ satisfies an 
inequality of the form
$$|SG| \leq |\psi(x,y,t,|u|) \, Su| + |\tilde{\psi_1}(x,y,t,|u|) \,S u| 
+  |\tilde{\psi_2}(x,y,t,|u|) \, u|$$
where $\tilde{\psi_1}(x,y,t,|u|) = u \partial_4 \psi $ and 
$\tilde{\psi_2}(x,y,t,|u|)= u \partial_{x_j} \psi$ or $u \partial_{y_j} \psi$
depending on $S=L_j$ or $M_j$. Moreover, by assumption (\ref{nlc2}) on $\psi$, 
we have $|\tilde{\psi_i}(x,y,t,|u|)| \leq C|u|^\alpha, ~ i=1,2.$
Hence, using arguments as in inequalities (\ref{est1psi}) and  (\ref{est2psi}), we get
\Bea && \left\Vert S G(z,t,|u(z,t)|) 
\right\Vert_{L^{q'}\left([-T,T];L^{p'}(\C^n)\right)} \\ \nonumber
&\leq& CT^{\frac{q-q\prime}{qq\prime}} 
\|u(z,t)\|_{L^\infty ([-T,T],{\tilde{W}^{1,2}(\C^n))} }^\alpha \times \nonumber \\
&& ( 2 \Vert S u(z,t)\Vert_{L^q[-T,T]:L^p(\C^n)} 
+ \Vert u(z,t)\Vert_{L^q[-T,T]:L^p(\C^n)} ) \nonumber\\
&\leq& CT^{\frac{q-q\prime}{qq\prime}} 
\|u(z,t)\|_{L^\infty ([-T,T],{\tilde{W}^{1,2}(\C^n))} }^\alpha
\Vert  u(z,t)\Vert_{L^q([-T,T]:\tilde{W}^{1,p}(\C^n))},  \nonumber
\Eea
which is the inequality (\ref{Gest2}).
\end{proof}

\begin{Proposition}\label{GEST}
 Let $G(z,t,w) $ be as in (\ref{nlc1}) and $\alpha$ as in assumption (\ref{nlc2}),  
$u \in L^\infty([-T,T];\tilde{W}^{1,2}(\C^n)) \cap 
{L^q([-T,T];\tilde{W}^{1,p}(\C^n))}$
where $(q,p)$ is an admissible pair with $p=\alpha +2$ and $q> 2$. Then for $|t| \leq T$, we have
\begin{equation}
\begin{aligned}\label{GEST1}
 \left\Vert\int_{0}^t  e^{-i(t-s)\mathcal{L}}
 G(z,s,u(z,s))ds\right\Vert_{L^q\left([-T,T] ;L^p(\C^n)\right)} \\
 &\hspace{-7cm} \leq CT^{\frac{q-q\prime}{qq\prime}}\|u(z,t)\|_
{L^\infty ([-T,T]; \tilde{W}^{1,2})} ^\alpha 
\| u(z,t)\|_{L^q ([-T,T],L^p(\C^n))},\\
\end{aligned}
\end{equation}

\begin{equation}
\begin{aligned}\label{GEST2}
 \left\Vert   \int_{0}^t   e^{-i(t-s)\mathcal{L}}\, G(z,s,u(z,s)) ds   
\right  \Vert_{L^q([-T,T];\tilde{W}^{1,p}(\C^n) )}  \\
 &\hspace{-7cm} \leq CT^{\frac{q-q\prime}{qq\prime}}\|u(z,t)\|_
{L^\infty ([-T,T]; \tilde{W}^{1,2})} ^\alpha \| u(z,t)\|_
{L^q ([-T,T],\tilde{W}^{1,p}(\C^n))}.
 \end{aligned}
\end{equation}
\end{Proposition}

\begin{proof} It follows from Proposition \ref{Gest}, that 
$G(z,t,|u(z,t)|) $ and $SG(z,t,|u(z,t)|)$  are in 
$L^{q'}\left([-T,T];L^{p'}(\C^n)\right)$
for admissible pairs $(q,p) $ with $p= \alpha +2$, $q>2$.
Since $(q,p)$ is admissible, by estimates (\ref{stauxest}) 
and (\ref{Gest1}), we get
\bea \label{pqineq1} 
&&  \left\Vert\int_{0}^t  e^{-i(t-s)\mathcal{L}} G(z,s,u(z,s))ds  
\right\Vert  _{L^{p,q}}  \nonumber \\
&& \hspace{1cm} \leq C\left\Vert G(z,s,|u(z,s)|) \right\Vert_{L^{p',q'}} \nonumber \\
  && \hspace{1cm} \leq CT^{\frac{q-q\prime}{qq\prime}}\|u(z,t)\|_{L^\infty 
([-T,T]; \tilde{W}^{1,2})} ^\alpha \| u(z,t)\|_{L^{p,q}}
\eea
which is inequality (\ref{GEST1}).

Again by commutativity of $L_j$ and $M_j$ with $\int_0^ t 
e^{-i(t-s)\LL} ds$ and the fact that
$SG(z,s,|u(z,s)|) \in L^{q'}\left([-T,T];L^{p'}(\C^n)\right)$, for $S=L_j, M_j, j=1,2,\ldots,n$,
we get as above
\bea \label{pqineq2}
&&  \left\Vert S \int_{0}^t  e^{-i(t-s)\mathcal{L}} G(z,s,u(z,s))ds  
\right\Vert  _{L^{p,q}} \nonumber\\
&& \hspace{1cm} \leq C\left\Vert SG(z,s,|u(z,s)|) 
\right\Vert_{L^{p', q'}} \nonumber\\
  && \hspace{1cm} \leq CT^{\frac{q-q\prime}{qq\prime}}\|u(z,t)\|_
{L^\infty ([-T,T]; \tilde{W}^{1,2})} ^\alpha 
\| u(z,t)\|_{L^q ([-T,T],\tilde{W}^{1,p}(\C^n))}
\eea and the inequality (\ref{GEST2}) follows from the above two estimates.
\end{proof}

\begin{Proposition}\label{GESTfor2}
Let $G(z,t,w) ,\alpha , u $ be as in Proposition \ref{GEST} and $(q,p)$ an admissible pair with $p=\alpha +2$ and $q> 2$. 
Then for $|t| \leq T$, we have
\begin{equation}
\begin{aligned}\label{GESTfor2.1}
 \left\Vert\int_{0}^t  e^{-i(t-s)\mathcal{L}}
 G(z,s,u(z,s))ds\right\Vert_{L^\infty \left([-T,T] ;L^2(\C^n)\right)} \\
 &\hspace{-7cm} \leq CT^{\frac{q-q\prime}{qq\prime}}\|u(z,t)\|_
{L^\infty ([-T,T]; \tilde{W}^{1,2})} ^\alpha 
\| u(z,t)\|_{L^q ([-T,T],L^p(\C^n))}
\end{aligned}
\end{equation}

\begin{equation}
\begin{aligned}\label{GESTfor2.2}
 \left\Vert   \int_{0}^t   e^{-i(t-s)\mathcal{L}}\, G(z,s,u(z,s)) ds   
\right  \Vert_{L^\infty([-T,T];\tilde{W}^{1,2}(\C^n) )}  \\
 &\hspace{-7cm} \leq CT^{\frac{q-q\prime}{qq\prime}}\|u(z,t)\|_
{L^\infty ([-T,T]; \tilde{W}^{1,2})} ^\alpha \| u(z,t)\|_
{L^q ([-T,T],\tilde{W}^{1,p}(\C^n))}.
 \end{aligned}
\end{equation}

\begin{proof}
The proof follows exactly as in Proposition \ref{GEST}, by using (\ref{retstrest}) instead of (\ref{stauxest}) in inequalities (\ref{pqineq1}) and (\ref{pqineq2}). 
\end{proof}

\end{Proposition}

\section{Local existence}\label{local existence}

In this section we prove local existence of solutions 
in the  first order Sobolev space $\tilde{W}^{1,2}$.
We follow Kato's approach using Strichartz estimates.  The key step is to  identify a subset in 
$L^\infty([-T,T]; \tilde{W}^{1,2}(\mathbb{C}^n))$, for a suitable $T$, where the operator
${\mathcal H}$ is a contraction. We proceed as follows:

For given positive numbers  
$T $ and $M$, consider the set $E=E_{T,M}$ given by

\begin{equation*}
  E=\left\{u \in  L^\infty\left([-T,T];\tilde{W}^{1,2}\right) 
\cap L^q\left([-T,T],\tilde{W}^{1,p}\right) \left\vert
\begin{array}{l} \Vert u\Vert_{L^{\infty}\left([-T,T],\tilde{W}^{1,2}\right)} \leq M,\\
 \Vert u\Vert_{L^q\left([-T,T],\tilde{W}^{1,p}\right)}\leq
  M\end{array}\right.\right\}
\end{equation*}
Introduce a metric on $E$, by setting
$$d(u,v)= \Vert u-v\Vert_{L^ \infty\left([-T,T],L^2 \right)}
+\Vert u-v\Vert_{L^q\left([-T,T],L^p \right)}. $$
\begin{Proposition}
$(E,d)$ is a complete metric space.
\end{Proposition}

\begin{proof}
Let $\{u_m\}$ be a Cauchy sequence in $(E,d)$. By passing to a 
subsequence if necessary,  we see that $\{u_m(\cdot,t)\}$ is 
Cauchy in $L^2(\C^n, dz)$ as well as in $L^q(\C^n,dz)$ for 
almost all $t$. Going for a further subsequence and appealing 
to an almost everywhere convergence argument in $z$-variable, 
we conclude that they have the same almost everywhere limit, 
say $u(\cdot,t)\in L^q \cap L^2 (\C^n)$ for almost all $t$. 
We need to show that 
$ u \in L^\infty(I;\tilde{W}^{1,2}(\C^n)) \cap L^q(I; \tilde{W}^{1,p}(\C^n))$ 
with
$$ \max \{ \Vert u\Vert_{L^{2,\infty}}, \Vert L_ju\Vert_{L^{2,\infty}}, 
\Vert M_ju\Vert_{L^{2,\infty}}, j=1,2,...,n\} \leq M $$ and 
$$ \max \{ \Vert u\Vert_{L^{p,q}}, \Vert L_ju\Vert_{L^{p,q}},  
\Vert M_ju\Vert_{L^{p,q}},  j=1,2,...,n\}  \leq M .$$
Let $S=L_j$ or $M_j$ be as before  and $\varphi \in 
C_c^\infty([-T,T] \times \C^n)$. Then for fixed $t\in [-T,T]$, using the pairing $\langle, \rangle_z$ in the $z$-variable,
we see that 
\Bea
 | \langle u(\cdot,t), S^*\varphi (\cdot,t) \rangle_z | &\leq 
& | \langle (u-u_m)(\cdot,t), S^* \varphi (\cdot,t) \rangle |  
+| \langle Su_m(\cdot,t), \varphi (\cdot,t) \rangle|\\
&\leq & \| u(\cdot,t)-u_m(\cdot,t)\|_{L^p (\C^n,dz)} \| 
S^* \varphi (\cdot,t)\|_{L^{p'} (\C^n,dz)}  \\
&&\hspace{.3cm}+ ~ \| Su_m(\cdot,t)\|_{L^p (\C^n,dz)} 
\| \varphi (\cdot,t)\|_{L^{p'} (\C^n,dz)}.
\Eea
Integrating with respect to $t$, and applying the H\"{o}lder's inequality in the $t$-variable, this yields
$$|\langle Su, \varphi  \rangle_{z,t}| \leq \| u-u_m\|_{L^{p,q}}  
\| S^* \varphi \|_{L^{p',q'}}+  \| Su_m\|_{L^{p,q}} \| \varphi \|_{L^{p',q'}}.$$
 Since $u_m \in E,~ \| Su_m \| _{L^{p,q}} \leq M$, thus
 letting  $m \rightarrow \infty $ we get
 \begin{equation*}
|\langle Su, \varphi  \rangle_{z,t}| 
\leq \underset{m \to \infty}{\limsup}\| Su_m\|_{L^{p,q}} 
\| \varphi \|_{L^{p',q'}} \leq M \| \varphi \|_{L^{p',q'}}.
\end{equation*}
Taking supremum over all $\varphi \in C_c^\infty ([-T,T]\times \C^n)$ 
with $ \| \varphi\|_{L^{p',q'}} \leq 1 $  this gives
\begin{equation*}
 \| Su\|_{L^q\left([-T,T];L^p\right)} \leq M.
\end{equation*}

For the pair $(\infty,2)$, take $\varphi \in C_c^\infty(\C^n)$, 
and by the same arguments as before
 \begin{equation*}
|\langle Su(\cdot,t), \varphi \rangle_z | \leq 
\underset{m \to \infty}{\limsup}\| Su_m(\cdot,t)\|_{L^2(\C^n)} \| 
\varphi \|_{L^2(\C^n)}
\end{equation*}
for almost every $t \in [-T,T]$. Taking supremum over all 
$\varphi \in C_c^\infty (\C^n)$ with $ \| \varphi\|_{L^2} 
\leq 1 $ this gives
\begin{equation*}
 \| Su(\cdot,t)\|_{L^2(\C^n)} \leq \underset{m \to \infty}{\limsup} 
\|Su_m\|_{L^{2,\infty}} \leq M .
\end{equation*}
Taking the essential supremum over $t \in [-T,T]$, we get the desired estimate.
\vspace{.4cm}
\end{proof}

\begin{proof}(of Theorem \ref{thm1})
We employ the traditional method of using contraction mapping theorem.  
Let $u \in E$. From (\ref{operator}) and the estimates in Lemma \ref{FEST}  and 
Proposition \ref{GEST}, we see that for all admissible pairs $(q,p)$, \Bea 
\| {\mathcal H} u\|_{L^q \left([-T,T],\tilde{W}^{1,p}\right)}\\ 
&\hspace{-3.6cm}\leq  \| e^{-it\mathcal{L}}f (z) \|_{L^q ([-T,T], \tilde{W}^{1,p})}+
\left \| \int_{0}^t e^{-i(t-s) \, \mathcal{L}}G(z,s,u(z,s)) \, ds 
\right \|_{L^q ([-T,T], \tilde{W}^{1,p})}\\ 
&\hspace{-13.4cm} \leq  C \, \| f\|_{\tilde{W}^{1,2}}\\
&\hspace{-4.3cm} +C \, 
T^{\frac{q-q\prime}{qq\prime}}\|u(z,t)\|_{L^\infty ([-T,T]; \tilde{W}^{1,2})} ^\alpha 
\| u(z,t)\|_{L^q ([-T,T],\tilde{W}^{1,p}(\C^n))} .\Eea
Clearly, this quantity is at most $M$, provided $T \leq \left( \frac{M - C\|f\|_{\tilde{W}^{1,2}}}{CM^{1+\alpha }}
\right)^{\frac{qq\prime} {q-q\prime}}$.

The above inequalities are also valid for the pair $(\infty,2)$ on LHS,
 by using Lemma \ref{FEST} and Proposition \ref{GESTfor2}. 
It follows that ${\mathcal H} u \in E=E_{T,M}$ provided  
\bea \label{T rel M} T \leq T_0 :=\left( \frac{M - C\|f\|_{\tilde{W}^{1,2}}}{CM^{1+\alpha }}
\right)^{\frac{qq\prime} {q-q\prime}}.\eea 

Now we show that  for a given $M$,
${\mathcal H}: E_{T,M} \to E_{T,M}$ is a contraction for  small $T$, i.e., for $T \leq \lambda T_0$ for some $\lambda<1$.

Let $u,v \in E_{T,M}$ with $T$ and $M$ as in $(\ref{T rel M})$. By mean value theorem on $\psi$ we see that
\bea \label{gdifr}|G(z,s,u) - G(z,s,v)| \leq |u-v| \, \Psi(u,v)\eea where
$\Psi(u,v)=\left( |w \partial_4\psi (x,y,s,w) | + |\psi(x,y,s,w) | 
\right) | _{w=\theta |u| + (1-\theta)|v| } $ for some $0<\theta<1$. 
Notice that in view of the condition (\ref{nlc2}) on $\psi$, 
$|\Psi(u,v)| \leq C(|u| \vee |v|)^\alpha$. Thus an application of H\"{o}lder's 
inequality in $z$-variable, 
using the relation $\frac{p'}{p}+\frac{\alpha p'}{p}=1$, followed by 
Sobolev embedding  result (Lemma \ref{SobL}) give
\bea \label{gdifraux}
\left \|   (u-v) \, \Psi(u,v) \right\|_{L^{p',q'}}
&\leq& C\| (|u| +|v|)^\alpha(u-v)\|_{L^{p',q'}} \nonumber \\
&\leq& CT^{\frac{q-q\prime}{qq\prime}}\, \|(|u|+ |v|)\|_{L^\infty(I, \tilde{W}^{1,2})}^\alpha 
\|u-v\|_{L^{p,q} } \nonumber\\
\label{psidif} &\leq& CT^{\frac{q-q\prime}{qq\prime}}\, M^\alpha \|u-v\|_{L^{p,q} } \eea
for  $2< q <\infty.$
In view of the estimates (\ref{stauxest}), (\ref{gdifr}) and 
(\ref{gdifraux}), the above leads to
\bea \label{contra1}  \left \| \int_0^t e^{-i(t-s){\mathcal
L}}\left[ G(u)-G(v)\right] (z,s)ds \right \|_{L^{p,q}} \nonumber 
&\leq& C\left \|   G(z,s,u)-G(z,s,v) \right\|_{L^{p',q'}}  \nonumber\\
 &\leq& C\, T^{\frac{q-q\prime}{qq\prime}} M^\alpha \|u-v\|_{L^{p,q} } \eea
for admissible pairs $(q,p)$. Similarly, using (\ref{retstrest}) 
instead of (\ref{stauxest}) , we also get
\bea \label{contra2}
 \left \| \int_0^t  e^{-i(t-s){\mathcal L}} [ G(u) - G(v) ] (z,s)ds 
\right \|_{L^{2,\infty}} \nonumber \\
 &\leq& C \, T^{\frac{q-q\prime}{qq\prime}} M^\alpha \|u-v\|_{L^{p,q} } .\eea
We choose $T= \lambda T_0$, for some $\lambda <1$. Thus from (\ref{contra1}) and (\ref{contra2}), we have
\bea \label{lambdacontra}d({\mathcal H}(u),{\mathcal H}(v))\leq 2C \, (\lambda T_0)^{\frac{q-q\prime}{qq\prime}} M^\alpha \|u-v\|_{L^{p,q} }. \eea

Taking $\lambda $ sufficiently small, we can make $k= 2 C \, (\lambda T_0)^{\frac{q-q\prime}{qq\prime}} M^\alpha
<1$ and we get $$d({\mathcal H}(u),{\mathcal H}(v))\leq k \, d(u,v).$$
This shows that ${\mathcal H}: E_{T,M} \to E_{T,M}$ is a contraction for $T=\lambda T_0$ and hence 
$\mathcal {H}$ has a unique fixed point in $E= E_{\lambda T_0, M}$. 

Note that we can fix a choice for $M$ and $T$ as follows: In view of the relation (\ref{T rel M}) between $T_0$  and $M$, when
$f \neq 0$, we can choose $M=2C\|f\|_{\tilde{W}^{1,2}}$ and $T=\lambda T_0 =  \lambda C_1\Vert f\Vert_{\tilde{W}^{1,2}}^{-\alpha \frac{qq\prime} 
{q-q\prime}}$
with $C_1=  (2C)^{-(1+\alpha){\frac{qq\prime} {q-q\prime}}}$, for any $\lambda<1$, i.e.,
\bea \label{lengthT} M=2C\|f\|_{\tilde{W}^{1,2}}~ \mbox{and}~T< (2C)^{-(1+\alpha){\frac{qq\prime} {q-q\prime}}}\Vert f\Vert_{\tilde{W}^{1,2}}^{-\alpha \frac{qq\prime} 
{q-q\prime}} ~~ \mbox{if} ~f \neq 0 .\eea
When $f\equiv 0$, $M$ can be any nonnegative number so that $T_0= (CM^\alpha)^{\frac{qq\prime} {q\prime-q}} $ and $T=\lambda T_0$ will work for any $\lambda <2^{\frac{qq\prime} {q\prime-q}}$. In particular, we can take   
\bea \label{lengthT2} M=1 \mbox{~and~}T< (2C)^{\frac{qq'}{q'-q}} 
\mbox{~if~} f \equiv 0.\eea 

Continuity: We will prove that $u\in C\left([-T,T];\tilde{W}^{1,2}
(\C^n)\right)$. Let $t_m\rightarrow t$ and set $u_m=u(z,t_m)$,  $S= L_j, M_j$ or the identity operator as before. Then
\Bea 
S(u_m-u)= e^{-it_m{\mathcal L}}Sf(z)-e^{-it{\mathcal L}}Sf(z) -i \int_0^{t_m} 
e^{-i(t_m-s){\mathcal L}}SG(z,s,u(z,s))ds\\
&\hspace{-9cm} +i \int_0^{t} e^{-i(t-s){\mathcal L}}SG(z,s,u(z,s))ds.
\Eea
Clearly $(e^{-it_m{\mathcal L}}-e^{-it{\mathcal L}})Sf(z)\rightarrow 0$ 
 in $L^2(\C^n)$ as $t_m\rightarrow t$.
 
 To deal with the other terms, we take $h \in L^2(\C^n)$ and estimate 
\begin{equation*}
\begin{aligned}
\left|\left\langle \int_0^{t} e^{-i(t_m-s){\mathcal L}}SG(z,s,u(z,s))ds-\int_0^{t} 
e^{-i(t-s){\mathcal L}}SG(z,s,u(z,s))ds,h \right\rangle\right| &\\
& \hspace{-10cm}
 =\left|\int_0^{t} \left\langle e^{-i(t_m-s){\mathcal L}}SG(z,s,u)-e^{-i(t-s){\mathcal L}}
SG(z,s,u),h \right\rangle ds\right| \\
&\hspace{-10cm} =\left|\int_0^{t} \left\langle SG,(e^{i(t_m-s){\mathcal L}}-
e^{i(t-s){\mathcal L}})h\right\rangle ds\right|\\
&\hspace{-10cm} \leq \Vert SG \Vert_{L^{p',q'}} \Vert e^{-is{\mathcal L}}
(e^{it_m{\mathcal L}}h-e^{it{\mathcal L}}h)\Vert_{L^{p,q}}\\
&\hspace{-10cm} \leq C\Vert SG \Vert_{L^{p',q'}} \Vert 
(e^{it_m{\mathcal L}}h-e^{it{\mathcal L}}h)\Vert_{L^2}. 
\end{aligned}
\end{equation*}
This shows that $\int_0^{t} e^{-i(t_m-s){\mathcal L}}SG(z,s,u(z,s))ds 
\rightarrow \int_0^{t} e^{-i(t-s){\mathcal L}}SG(z,s,u(z,s))ds$ 
weakly in $L^2(\C^n)$. This sequence is also bounded in $L^2$:
\Bea \left\Vert\int_0^{t} e^{-i(t_m-s){\mathcal L}}SG(z,s,u(z,s))ds\right\Vert_{L^2(\C^n)} &\\
&\hspace{-7cm}=\left\Vert e^{-i(t_m-t){\mathcal L}}\int_0^{t} e^{-i(t-s){\mathcal L}}
SG(z,s,u(z,s))ds\right\Vert_{L^2(\C^n)}\\
&\hspace{-8.8cm}=\left\Vert\int_0^{t} e^{-i(t-s){\mathcal L}}SG(z,s,u(z,s))ds\right\Vert_{L^2(\C^n)}\Eea
which is finite by Proposition \ref{GESTfor2}. 
Thus we have the  convergence  in $L^2(\C^n)$: $$\int_0^{t} e^{-i(t_m-s){\mathcal L}}SG(z,s,u(z,s))ds\rightarrow  
\int_0^{t} e^{-i(t-s){\mathcal L}}SG(z,s,u(z,s))ds.$$
Also since
$\left\Vert \int_t^{t_m} e^{-i(t_m-s){\mathcal L}}SG(z,s,u(z,s))ds
\right\Vert_{L^2}\leq C\Vert SG \Vert_{L^{q'}([t,t_m],L^{p'})}\rightarrow 0$ 
as $t_m\rightarrow t$, we conclude that $u(z,t_m)
\rightarrow u(z,t)$ in $\tilde{W}^{1,2}$.
\end{proof}

\begin{Remark} \label{Rk4.2}
The above proof also shows that if we consider the initial value problem with an arbitrary initial time $t_0$, then the solution exists on an 
interval $[t_0-T, t_0 + T]$ with $T$ given by the same inequalities as in (\ref{lengthT}) and (\ref{lengthT2}) 
but with $\|f\|_{\tilde{W}^{1,2}}$ replaced by $\| u(.,t_0)\|_{\tilde{W}^{1,2}}$.
\end{Remark}
The following two results are used in the proof of Theorem  \ref{wellposedness}.

\begin{Proposition}\label{forstability}
 Let $\Phi$ be a continuous complex valued function on $\C$ such that $|\Phi(w)|\leq C|w|^{\alpha}$ for $0\leq \alpha<\frac{2}{n-1}$.
 Suppose $\{ u_m\} $ be a sequence in 
 $L^q\left([a,b],\tilde{W}^{1,p}\right)\cap L^{\infty}([a,b],\tilde{W}^{1,2})$, $p=2+ \alpha, ~ q \geq 2$, such that 
$$\sup_{m\in \mathbb{N}} \Vert u_m\Vert_{L^{\infty}\left([a,b],\tilde{W}^{1,2}\right)} \leq M <\infty.$$
If $u_m\rightarrow u$ in $L^q([a,b],L^p(\C^n))$  then 
$[\Phi(u_m)-\Phi(u)]Su \rightarrow 0$ in $L^{q'}\left([a,b],L^{p'}(\C^n) \right),$ 
for  $S=I,L_j,M_j;1\leq j\leq n$.
\end{Proposition}

\begin{proof}
Since $u_m\to u$ in $L^q([a,b],L^p(\C^n))$, we can extract a subsequence 
still denoted by ${u_k}$ such that 
$$ \Vert u_{k+1}-u_{k}\Vert_{L^q\left([a,b],L^p(\C^n)\right)}\leq \frac{1}{2^k}$$ 
for all $k\geq 1$ and $u_k(z,t)\to u(z,t)$ a.e. Hence by continuity of $\Phi$,
\bea [\Phi(u_k)-\Phi(u)]Su \rightarrow 0 \hskip.1in~\mbox{for a.e} ~(z,t)\in \C^n \times [a,b] .\eea
We establish the norm convergence by appealing to a dominated convergence argument in 
$z$ and $t$ variables successively.

Consider the function  $H(z,t)=\sum_{k=1}^\infty |u_{k+1}(z,t)-u_k(z,t)|$. Clearly
  $H \in L^q\left([a,b],L^p(\C^n)\right)$, since the above series converges absolutely in
  that space. Also for $l>k$, $ |(u_l-u_k)(z,t)|\leq |u_l-u_{l-1}|+\cdots +|u_{k+1}-u_k|\leq H(z,t)$
 hence $ |u_k -u| \leq H.$
This leads to the pointwise almost everywhere inequality \Bea \label{dominate} |u_k(z,t) | \leq |u(z,t)| + H(z,t)=v(z,t).\Eea
Hence $$ |\left[\Phi(u_k)-\Phi(u)\right]Su(z,t)|^{p'} \leq  | [v^{\alpha}+|u|^{\alpha} ]  Su(z,t)|^{p'} .$$
Since $u, v \in L^q\left([a,b],L^p(\C^n)\right)$ and $p=2 + \alpha$, using H\"{o}lder's inequality with $\frac{p'}{p}+\frac{\alpha p'}{p}=1$, 
 we get
\bea \label{4.6}
  \int_{\C^n} |(v^{\alpha}+|u|^{\alpha})  Su(z,t)|^{p'} dz \\
\nonumber  \leq ( \| v(\cdot,t) \|^ {\alpha p'}_{L^p(\C^n)} \!\!\!\!\! &+& \!\!\!\!  \|u(\cdot,t) \|^ {\alpha p'} _{L^p(\C^n)}  ) \| Su(\cdot,t)\|_{L^p(\C^n)} ^{p'}.
 \eea

Thus using dominated convergence theorem in $z$-variable, we see that 
 \bea \label{4.8}\Vert \left[\Phi(u_k)-\Phi(u)\right]Su(\cdot,t)\Vert_{L^{p'}(\C^n)}
\rightarrow 0\eea
 as $k\rightarrow \infty,$ for a.e. $t$.

Again, in view of  Lemma \ref{embedding}, and H\"{o}lder's inequality as above, we get
\Bea
 \| [\Phi(u_k)&-&\Phi(u)]\, Su(\cdot,t) \|_{L^{p'}(\C^n)}\\ 
 &\leq& C\left(\| u_k\|_{L^{\infty}\left([a,b],\tilde{W}^{1,2}\right)}^{\alpha}+
\|u\|_{L^{\infty}\left([a,b],\tilde{W}^{1,2}\right)}^{\alpha}\right)\| Su(\cdot,t)\|_{L^{p}}\\
&\leq& C(M^\alpha + \|u\|_{L^{\infty}\left([a,b],\tilde{W}^{1,2}\right)}^{\alpha}) \| Su(\cdot,t)\|_{L^{p}}.
\Eea

Since $\|Su(\cdot,t) \|_{L^p(\C^n)} \in L^{q'}([a,b])$ and $q \geq 2$, an application of the H\"{o}lder's inequality in the $t$-variable shows that 
$$\int_a^b \|Su(\cdot,t) \|_{L^p(\C^n)} ^{q'} dt \leq [b-a]^{\frac{q-q'}{q} } \,  \|Su(\cdot,t) \|_{L^q([a,b], L^p(\C^n))} ^{q'}.$$
Hence a further application of dominated convergence theorem in 
(\ref{4.8}) shows  that $\Vert \left(\Phi(u_k)-\Phi(u)\right)Su\Vert_{L^{q'}([a,b],L^{p'})}\rightarrow 0$, as
$k\rightarrow \infty$. 

Thus we have shown that $\left[\Phi(u_{m_k})-\Phi(u)\right]Su\rightarrow 0$
in $L^{q'}([a,b],L^{p'}(\C^n))$
whenever  $u_m \to u$ in $L^{q}([a,b],L^{p}(\C^n))$. But the above arguments are also valid if we had started with any subsequence of $u_m$. 
It follows that any subsequence of $\left[\Phi(u_{m})-\Phi(u)\right]Su$ has a  subsequence that converges to $0$
in $L^{q'}([a,b],L^{p'}(\C^n))$. From this we conclude that the original sequence \linebreak $\left[\Phi(u_m)-\Phi(u)\right]Su$ converges to zero in $L^{q'}([a,b],L^{p'}(\C^n))$, hence the proposition.
  \end{proof}

\begin{Proposition}\label{forstability2} Let  $\{f_m\}_{m\geq 1}$ be a sequence in 
$ \tilde{W}^{1,2}(\C^n)$ such that $f_m\rightarrow f$ in $\tilde{W}^{1,2}(\C^n)$ as $m \rightarrow \infty$. 
Let  $u_m$ and $u$ be the solutions corresponding to the initial data $f_m$ and $f$ respectively, at time $t=t_0$. Then there exists $\tau_0$, depending on $\Vert f\Vert_{\tilde{W}^{1,2}}$ such that  $\Vert u_m-u\Vert_{L^{\infty}([t_0,t_0+\tau_0],\tilde{W}^{1,2}(\C^n))} \rightarrow 0$.

\end{Proposition}

\begin{proof}
 Let $\epsilon>0$ and $\tilde{\tau}<\lambda T_0$, where $T_0$ is as in  (\ref{T rel M}) and $\lambda$ as in (\ref{lambdacontra}).
 Since the time interval of existence is given by $ \tilde{W}^{1,2}(\C^n)$ norm of the initial data and since
$\| f_m\|_{ \tilde{W}^{1,2}(\C^n)} \to\| f\|_{ \tilde{W}^{1,2}(\C^n)}  $
we can assume, by taking $m$  large if necessary, that both the 
solutions $u$ and $u_m$ are defined on $[t_0,t_0+\tilde{\tau}]$. Setting 
$G_m(z,t)=G(z,t,u_m(z,t))$, we have
\Bea
(u_m-u)(z,t)=e^{-i(t-t_0)\LL}(f_m-f)(z,t)-i
\int_{t_0}^t e^{-i(t-s)\LL}(G_m-G)(z,s)ds
\Eea
for all $t \leq \tilde{\tau}$. In view of  Lemma \ref{commutativity}, with $S=I, L_j,M_j$, we also have 
\begin{equation}
\begin{aligned}\label{}
S(u_m-u)(z,t)=e^{-i(t-t_0)\LL}S(f_m-f)(z,t) \nonumber\\
& \hspace{-3cm}-i\displaystyle\int_{t_0}^t e^{-i(t-s)\LL}S(G_m-G)(z,s)ds. \nonumber
\end{aligned}
\end{equation}
Thus by estimates in Theorem \ref{strichartzs}, we see that for any $\tilde{\tau_0}\leq \tilde{\tau}$
\begin{eqnarray}
\label{stability1}
\hspace{1cm}\|S(u_m-u)(z,t)\|_{L^{q}([t_0,t_0+\tilde{\tau_0}],L^p)} &\leq& C\|S(f_m-f)\|_2 \\
&& + C\|S(G_m-G)\|_{L^{q'}([t_0,t_0+\tilde{\tau_0}],L^{p'})} \nonumber
\end{eqnarray} 
for $S=L_j, M_j$ and $I$, the identity operator for admissible pairs $(q,p)$.

First we consider the case $f\equiv 0$. In this case the solution $u\equiv 0$ since ${\mathcal H}(0)=0$ and the fixed point of ${\mathcal H}$ in $E$ is unique.
 Thus by estimates in Theorem \ref{strichartzs}, we see that for any $\tilde{\tau_0}\leq \tilde{\tau}$
\begin{eqnarray}
\label{stability10}
\hspace{1cm}\|Su_m(z,t)\|_{L^{q}([t_0,t_0+\tilde{\tau_0}],L^p)} &\leq& C\|Sf_m\|_2 \\
&& + C\|SG_m\|_{L^{q'}([t_0,t_0+\tilde{\tau_0}],L^{p'})} \nonumber
\end{eqnarray} 
for $S=L_j, M_j$ and $I$, the identity operator.
For $S=I$, using estimate (\ref{Gest1}), we get
\begin{eqnarray*}
\|u_m(z,t)\|_{L^{q}([t_0,t_0+\tilde{\tau_0}],L^p)} &\leq& C\|f_m\|_2 \\
&& + C\tilde{\tau}^{\frac{q-q\prime}{qq\prime}}\|u_m\|_
{L^\infty ([t_0,t_0+\tilde{\tau_0}]; \tilde{W}^{1,2})} ^\alpha 
\| u_m\|_{L^q ([t_0,t_0+\tilde{\tau_0}],\tilde{W}^{1,p})}. \nonumber
\end{eqnarray*} 

Now we observe that $\| u_m\|_{L^q ([t_0,t_0+\tilde{\tau_0}],\tilde{W}^{1,p})}$ is uniformly bounded. In fact by choice of
$\tilde{\tau_0}$, and from the local existence theorem proved above, we have $\| u_m\|_{L^q ([t_0,t_0+\tilde{\tau_0}],\tilde{W}^{1,p})}\leq M_m$ which is given by equations (\ref{lengthT}) and (\ref{lengthT2}), for each $m$. Since

\Bea
M_m= \left\{ 
  \begin{array}{cc}
    1 & \mbox{~if~} f_m=0 \\ 
    2C \| f_m\|_{\tilde{W}^{1,2}} & \mbox{~if~} f_m \neq 0\\ 
  \end{array}
\right.
\Eea
and  $ \| f_m\|_{\tilde{W}^{1,2}} \to  \| f\|_{\tilde{W}^{1,2}}=0$, we have $M_m \leq 1$ for large $m$.

Now choosing $\tilde{\tau_0}$ a value of $\tilde{\tau}$ small so that $C \tilde{\tau}^{\frac{q-q'}{qq'}} < \frac{1}{2}$, 
we see that
\begin{eqnarray}\label{convergenceinLpq10}
\Vert u_m\Vert_{L^q([t_0,t_0+\tilde{\tau_0}],L^p(\C^n))}  \leq 2C\Vert f_m\Vert_{L^2(\C^n)}\rightarrow 0 ~\mbox{as}~ m\rightarrow \infty.
\end{eqnarray}
Thus in view of  estimate (\ref{GESTfor2.1})  and estimate (\ref{retstrest}) in Theorem \ref{strichartzs} we see that $u_m\rightarrow 0$ in $L^{\infty}([t_0,t_0+\tilde{\tau_0}],L^2(\C^n))$.

Similarly, using the estimates (\ref{Gest2})  and (\ref{stability10}), we see that $Su_m \rightarrow 0$ in $L^q([t_0,t_0+\tilde{\tau_0}],L^p(\C^n))$. It follows from estimate (\ref{GESTfor2.2}) that
 $u_m \rightarrow 0$ in $L^{\infty}([t_0,t_0+\tilde{\tau_0}],\tilde{W}^{1,2}(\C^n))$.

Now we consider the case $f\neq 0$.
 We choose $m$ sufficiently large such that $\Vert f_m-f\Vert_{\tilde{W}^{1,2}} < \epsilon <
 \frac{1}{2}\Vert f\Vert_{\tilde{W}^{1,2}}$. Therefore we have 
 $ \Vert f_m\Vert_{\tilde{W}^{1,2}} \leq \frac{3}{2}\Vert f\Vert_{\tilde{W}^{1,2}}$
and hence again by (\ref{lengthT}), $M_m :=2C\Vert f_m\Vert_{\tilde{W}^{1,2}} \leq 2M:=4C\Vert f\Vert_{\tilde{W}^{1,2}}$. 
Now by (\ref{gdifr}),(\ref{gdifraux}), and the fact that $M_m\leq 2M$ we get
\begin{equation}\label{GmGestimate}
\|G_m-G\|_{L^{q'}([t_0,t_0+\tilde{\tau}],L^{p'}(\C^n))} \leq C \tilde{\tau}^{\frac{q-q'}{qq'}}
\Vert f\Vert_{\tilde{W}^{1,2}}^{\alpha} \Vert u_m-u\Vert_{L^q([t_0,t_0+\tilde{\tau}],L^p(\C^n))}.
\end{equation}
This gives for the case $S=I$, the inequality 
\begin{equation*}
\begin{aligned}
\Vert u_m-u\Vert_{L^q([t_0,t_0+\tilde{\tau}],L^p(\C^n))} &\\
&\hspace{-2.3cm}\leq C\Vert f_m-f\Vert_{L^2}+C \tilde{\tau}^{\frac{q-q'}{qq'}}
\Vert f\Vert_{\tilde{W}^{1,2}}^{\alpha} \Vert u_m-u\Vert_{L^q([t_0,t_0+\tilde{\tau}],L^p(\C^n))}.
\end{aligned}
\end{equation*}
Now choosing $\tilde{\tau_0}$ a value of $\tilde{\tau}$ small so that $C \tilde{\tau}^{\frac{q-q'}{qq'}}
\Vert f\Vert_{\tilde{W}^{1,2}}^{\alpha} < \frac{1}{2}$, we see that
\begin{eqnarray}\label{convergenceinLpq}
\Vert u_m-u\Vert_{L^q([t_0,t_0+\tilde{\tau_0}],L^p(\C^n))}  \leq C\Vert f_m-f\Vert_{L^2(\C^n)}< C \epsilon 
\end{eqnarray} for large $m$.
Thus in view of Theorem \ref{strichartzs}, (\ref{GmGestimate}) and (\ref{convergenceinLpq}) we see that $u_m\rightarrow u$ in $L^{\infty}([t_0,t_0+\tilde{\tau_0}],L^2(\C^n))$.

For $S=L_j, M_j$ using (\ref{L of upsi}), (\ref{M of upsi}) with the notation $\psi_m= \psi\left(z,t,|u_m(z,t)|\right)$, we have 
\begin{equation}
\begin{aligned}\label{stability}
S(G_m-G)=\psi_m S(u_m-u)+(\psi_m-\psi)Su+(\partial_j\psi_m) 
(u_m-u)\\
&\hspace{-8.6cm}+(\partial_j\psi_m-\partial_j\psi)u
+(\partial_4\psi_m) u_m\Re(\frac{\overline{u_m}}{|u_m|}S(u_m-u))\\
&\hspace{-8.6cm}+(\partial_4\psi_m) u_m\Re(\frac{\overline{u_m}}
{|u_m|}Su)-(\partial_4\psi) u\Re(\frac{\overline{u}}{|u|}Su)
\end{aligned}
\end{equation}
where $\partial_j= \partial_{x_j}$ for $S=L_j$ and 
$\partial_j= \partial_{y_j}$ for $S=M_j,~1\leq j\leq n$.

Using the assumption (\ref{nlc2}) on $\psi$,  Lemma \ref{embedding}, and the fact that $M_m\leq 2M$, similar computations 
as in Proposition \ref{Gest} shows that 
$$
\Vert\psi_m S(u_m-u)\Vert_{L^{q'}\left([t_0,t_0+\tilde{\tau_0}],L^{p'}\right)} \leq C{\tilde{\tau_0}}^{\frac{q-q'}{qq'}}\Vert f\Vert_{\tilde{W}^{1,2}}^{\alpha}\Vert S(u_m-u)
\Vert_{L^q\left([t_0,t_0+\tilde{\tau_0}],L^p\right)}
$$
$$
\hspace*{-1cm} \Vert(\partial_j\psi_m)(u_m-u)\Vert_{L^{q'}\left([t_0,t_0+\tilde{\tau_0}],L^{p'}\right)} 
\leq C{\tilde{\tau_0}}^{\frac{q-q'}{qq'}}\Vert f\Vert_{\tilde{W}^{1,2}}^{\alpha}\Vert u_m-u\Vert_{L^q\left([t_0,t_0+\tilde{\tau_0}],L^p\right)}
$$
\begin{equation*}
\begin{split}
&\Vert (\partial_4\psi_m) u_m\Re(\frac{\overline{u_m}}{|u_m|}S(u_m-u)) 
\Vert_{L^{q'}\left([t_0,t_0+\tilde{\tau_0}],L^{p'}\right)} \\
& \hspace{5cm} \leq  C{\tilde{\tau_0}}^{\frac{q-q'}{qq'}}\Vert f\Vert_{\tilde{W}^{1,2}}^{\alpha}\Vert S(u_m-u)\Vert_{L^q\left([t_0,t_0+\tilde{\tau_0}],L^p\right)}.
\end{split}
\end{equation*}

Since  $\Vert u_m-u\Vert_{L^q([t_0,t_0+\tilde{\tau_0}],L^p(\C^n))}\rightarrow 0$ and $G$ is $C^1$, so in view of the condition (\ref{nlc2}) on $\psi$ and Proposition \ref{forstability}, the sequences $(\psi_m-\psi)Su,(\partial_j\psi_m-\partial_j\psi)u$ and $(\partial_4\psi_m) u_m\Re(\frac{\overline{u_m}}{|u_m|}Su)-(\partial_4\psi) u\Re(\frac{\overline{u}}{|u|}Su)$ converge to zero 
in $L^{q'}([t_0,t_0+\tilde{\tau_0}],L^{p'})$ as $m\rightarrow \infty$. Hence for large $m$
\Bea \|(\psi_m-\psi)Su \|_{L^{q'}([t_0,t_0+\tilde{\tau_0}],L^{p'})} < \epsilon\\
\| (\partial_j\psi_m-\partial_j\psi)u \|_{L^{q'}([t_0,t_0+\tilde{\tau_0}],L^{p'})}  < \epsilon\\
\left \| (\partial_4\psi_m) u_m\Re
(\frac{\overline{u_m}}{|u_m|}Su-(\partial_4\psi) 
u\Re(\frac{\overline{u}}{|u|}Su)) \right\|_{L^{q'}([t_0,t_0+\tilde{\tau_0}],L^{p'})} < \epsilon.
\Eea
Using these estimates in (\ref{stability}), we get
\begin{eqnarray*}
\Vert S(G_m-G)\Vert_{L^{q'}([t_0, t_0+\tilde{\tau_0}],L^{p'})}  \leq &\!\!\!\!\!  C \tilde{\tau_0}^{\frac{q-q'}{qq'}}\Vert f\Vert_{\tilde{W}^{1,2}}^{\alpha}\Vert S(u_m-u)\Vert_{{L^{q}([t_0,t_0+\tilde{\tau_0}],L^{p}(\C^n))}} \\
& + C\tilde{\tau_0}^{\frac{q-q'}{qq'}}\Vert f\Vert_{\tilde{W}^{1,2}}^{\alpha}\Vert u_m-u\Vert_{L^{q}([t_0,t_0+\tilde{\tau_0}],L^{p})} + 3 \epsilon. \nonumber
\end{eqnarray*}
Thus from estimates (\ref{stability1}) and (\ref{convergenceinLpq}) we see that,
\Bea\Vert S(u_m-u)\Vert_{L^q([t_0,t_0+\tilde{\tau_0}],L^p)}\leq C\Vert S(f_m-f)
\Vert_{L^2}+C\Vert S(G_m-G)\Vert_{L^{q'}([t_0,t_0+\tilde{\tau_0}],L^{p'})}\\
&\hspace{-10.2cm} \leq  C\tilde{\tau_0}^{\frac{q-q'}{qq'}}
\Vert f\Vert_{\tilde{W}^{1,2}}^{\alpha}\Vert S(u_m-u)\Vert_{L^q([t_0,t_0+\tilde{\tau_0}],L^p)} \\
+ ~ C^2\tilde{\tau_0}^{\frac{q-q'}{qq'}}\Vert f\Vert_{\tilde{W}^{1,2}}^{\alpha} \epsilon+(C+3) \epsilon
\Eea
for large $m$. Now choose $\tau_0< \tilde{\tau_0},$ so that 
$$  C\tau_0^{\frac{q-q'}{qq'}}
\Vert f\Vert_{\tilde{W}^{1,2}}^{\alpha}\leq\frac{1}{2}$$
and we see that
$\Vert S(u_m-u)\Vert_{L^q([t_0,t_0+\tau_0],L^p)}\leq (3C+6)\epsilon$ for large $m$. This shows that
$\Vert S(u_m-u)\Vert_{L^q([t_0,t_0+\tau_0],L^p)} \rightarrow 0$ as $m\rightarrow \infty$.

\end{proof}
\begin{Remark}\label{forstability3} The conclusions of Proposition \ref{forstability2} is also valid for the left side interval $[t_0-\tau_0,t_0]$.
\end{Remark}
\begin{proof}(of Theorem \ref{wellposedness}) 
By local existence (Theorem 1), the solution exists in  $C([-T,T]: \tilde{W}^{1,2}(\C^n))$. Now, if $\| u( \cdot,T)\|_{\tilde{W}^{1,2}(\C^n)}< \infty,$ the argument in the proof of Theorem 1 can be carried out with $T$ as the initial time, to extend
the solution to the interval $[T,T_1]$. This procedure can be continued and we can get a sequence $\{T_j\}$ such that
$T<T_1<T_2 <\dots <T_n <...$ as long as $\| u( \cdot,T_j)\|_{\tilde{W}^{1,2}(\C^n)}< \infty.$  Let $T_+= \displaystyle{\sup_j}~ T_j$ so that the solution extends to $[0,T_+)$. In the same way we can extend the solution to the left side to the interval $(T_-,0]$ to get a solution in $C((T_-,T_+), \tilde{W}^{1,2}(\C^n))$.
This leads to the following proof:

Blowup alternative:
Suppose  $T_+<\infty$ and 
$ \underset{ t \to T_+}{\lim}\Vert u(z,t)\Vert_{\tilde{W}^{1,2}} =M_0< \infty$.
Then we can choose a sequence $t_j \uparrow T_+$ such that 
$\Vert u(z,t_j)\Vert_{\tilde{W}^{1,2}} \leq M_0$. From local existence
and in view of Remark \ref{Rk4.2}
we can choose $T_j = \frac{1}{10} \, C_1\| u(.,t_j)\|_{\tilde{W}^{1,2}} ^{-\alpha \frac{qq\prime} {q-q\prime}}$ such that $u \in C([t_j- T_j, t_j + T_j], 
\tilde{W}^{1,2})$. Hence by assumption  
$T_j \geq  \frac{1}{10} C_1M_0^{-\alpha \frac{qq\prime} {q-q\prime}}$
, a constant independent of $t_j$, for $q > 2$. Thus we can choose 
$j$ so large that $t_j+T_j>T_+$, which contradicts maximality of $T_+$. 
Hence  if $T_+< \infty$,
$ \underset{t \to T_+}{\lim}\Vert u(z,t)\Vert_{\tilde{W}^{1,2}} =\infty.$ 
Similarly, we can show that 
$\underset{t\to T_-}{\lim} \Vert u(.,t)\Vert_{\tilde{W}^{1,2}} = \infty$, 
if $T_- > - \infty$.

Uniqueness: Let $(T_-,T_+)$ be the maximal interval 
about zero, such that the solution $u(x,t)$ exists in $C((T_-,T_+), 
\tilde{W}^{1,2})$ and in Theorem \ref{thm1}, we have already shown that the 
solution exists on $[-T,T] \subset (T_-,T_+)$. We first consider the subinterval  
$[0,T] \subset [0,T_+)$. 

Suppose that there exist 
two solutions $u$ and $v$ of equations (\ref{mainpde}) and (\ref{inidata}) on $[0,T_+)$.
In particular for $\tau \in (T, T_+),$ we have
\Bea u(z,\tau)&=& e^{-i(\tau-T){\mathcal L}} u(z,T) -i \int_T^\tau e^{-i(\tau-s){\mathcal
L}}G(z,s,u(z,s))ds,\\
v(z,\tau)&=& e^{-i(\tau-T){\mathcal L}} v(z,T) -i \int_T^\tau e^{-i(\tau-s){\mathcal
L}}G(z,s,v(z,s))ds.\Eea
Since the solution given by the contraction mapping
is continuous and unique on $[0,T] \subset [0,T_+)$, we have $u(z, T)=v(z,T)$. 
Hence using an inequality as in (\ref{contra1}), this leads to
\Bea \| u-v\|_{L^q\left([T,\tau],L^p(\C^n)\right)} &= &  \left \| \int_T^\tau e^{-i(t-s){\mathcal
L}}\left(G(u)-G(v)\right)(z,s)ds \right \|_{L^q\left([T,\tau],L^p(\C^n)\right)}\\
&\leq& C[\tau-T]^{\frac{q-q'}{qq'}} \, M_{T,\tau}^\alpha \|u-v\|_{L^q\left([T,\tau],L^p(\C^n)\right)}\Eea
for all $ \tau \in [T,T_+)$ where $M_{T,\tau}=\max \{\Vert u\Vert_{L^{\infty}
\left([T,\tau],\tilde{W}^{1,2}\right)},\Vert v\Vert_{L^{\infty}\left([T,\tau],\tilde{W}^{1,2}\right)}\}$. 
Since $u,v\in C\left((T_-,T_+),\tilde{W}^{1,2}\right)$, we have
 $M_{T,\tau}<\infty$.
In particular choose $\tilde{T} \in [T,T_+)$ 
such that $C|\tilde{T}-T|^{\frac{q-q'}{qq'}} \, M_{T,\tilde{T}}^\alpha =c < 1$,  so that
$$0 \leq (1-c) \|u-v \|_{L^q([T,\tilde{T}]: L^p(\C^n))} \leq 0.$$
Hence $u=v$ on the larger interval $[0,\tilde{T}].$ 

Now let $\theta$ =sup$\{\tilde{T}: 0<\tilde{T}
<T_+:\Vert u-v\Vert_ {L^q([0,\tilde{T}],L^p)}=0\}$. 
If $\theta<T_+$, then by the above observation, 
$\Vert u-v\Vert_ {L^q([0,\theta+\epsilon],L^p)}=0$
for some $\epsilon>0$,which contradicts the definition 
of $\theta$. Thus we conclude that $\theta=T_+$, proving the uniqueness on $[0,T_+)$. 
Similarly one can show uniqueness on $(T_-,0]$.

Stability: Let  $\{f_m\}_{m\geq 1}$ be a sequence in 
$ \tilde{W}^{1,2}(\C^n)$ such that $f_m\rightarrow f$ in $\tilde{W}^{1,2}$ as $m \rightarrow \infty$. 
Let  $u_m$ and $u$ be the solutions corresponding to the initial data $f_m$ and $f$ respectively. 
Let $(T_-,T_+)$ and $(T_{m-},T_{m+})$ be maximal intervals for the solutions $u$ and
$u_m$ respectively and $I \subset (T_-,T_+)$  be a compact interval. 

The key idea is to extend  the stability result proved in Proposition \ref{forstability2} to the interval $I$ 
by covering it with finitely many intervals obtained by successive application of Proposition 
\ref{forstability2}. This is possible provided $u_m$ is defined on $I$, for all but finitely many $m$. 
In fact, we prove $I \subset (T_{m-},T_{m+})$ for all but finitely many $m$.

We can assume that $0\in I=[a,b]$, and give a proof by the method of contradiction. 
Suppose there exist infinitely many $T_{m_j +} \leq b$, let  
$c= \liminf T_{m_j+}$. Then for $\epsilon>0, ~ [0, c-\epsilon] \subset [0,T_{m_j+})$ for all  $m_j$ sufficiently large and $u_{m_j}$ are defined on $ [0, c-\epsilon]$. 

By compactness, the stability
result proved in Proposition \ref{forstability2} can be extended to the interval $[0, c-\epsilon]$ 
by covering it with finitely many intervals obtained by successive application of Proposition \ref{forstability2}.
Hence 
$$ \Vert u_{m_j}(.,c-\epsilon)\Vert_{\tilde{W}^{1,2}}\rightarrow  \Vert u(.,c-\epsilon)\Vert_{\tilde{W}^{1,2}} ~~\mbox{as} ~~ j\rightarrow \infty.$$ 
Also by continuity we have 
$$\Vert u(.,c-\epsilon)\Vert_{\tilde{W}^{1,2}} \rightarrow \Vert u(.,c)\Vert_{\tilde{W}^{1,2}}  ~~\mbox{as} ~~\epsilon\rightarrow 0.$$ 
Thus, for any $\delta>0$, we have 
\bea \label{deltaineq} \Vert u_{m_j}(.,c-\epsilon)\Vert_{\tilde{W}^{1,2}}^{-\alpha \frac{qq'}{q-q'}}> \delta 
\mbox{~whenever ~}\Vert u(.,c)\Vert_{\tilde{W}^{1,2}}^{-\alpha \frac{qq'}{q-q'}}>\delta,\eea  for sufficiently small $\epsilon$  and for all $j\geq j_0(\epsilon)$. 
Therefore by applying the local existence theorem, with $c-\epsilon$ as the initial time, we  see that  $u_{m_j}$ extends to 
 $[0, c-\epsilon +\frac{C_1}{10} \Vert u_{m_j}(.,c-\epsilon)\Vert_{\tilde{W}^{1,2}}^{-\alpha \frac{qq'}{q-q'}} ]$ for large $j$.
 Now choosing  $\epsilon<\frac{C_1}{20} \, \delta$, we have by (\ref{deltaineq})
$$c-\epsilon +\frac{C_1}{10} \Vert u_{m_j}(.,c-\epsilon)\Vert_{\tilde{W}^{1,2}}^{-\alpha \frac{qq'}{q-q'}}> c +\frac{C_1}{20} \delta~~~\mbox{for all}~~ j\geq j_0(\epsilon).$$
It follows that $T_{m_j+}\geq c+\frac{C_1}{20}\delta$, hence contradicts the fact that $ \liminf T_{m_j+}=c.$

Similarly we can show that
$[a,0] \subset (T_{m-},0]$ for all but finitely many $m$ which completes the proof of stability.
\end{proof}

\section{Appendix} In this section we show the equivalence of  
the differential equations (\ref{mainpde}), (\ref{inidata})
and the integral equation (\ref{solution}). 

\begin{Lemma}\label{equivalence} Let $u \in L^\infty(I, 
\tilde{W}^{1,2}(\C^n)) \cap L^q(I, \tilde{W}^{1,p}(\C^n))$ and $G$ as in (\ref{nlc1}), (\ref{nlc2}). 
Then $u$ satisfies the nonlinear Schr\"{o}dinger equation 
(\ref{mainpde}), with initial data  (\ref{inidata}) if and only if $u$ satisfies 
the integral equation (\ref{solution}).
\end{Lemma}
\begin{proof}  First observe that  the  following equalities 
\bea\label{derivativef}
\frac{\partial}{\partial_t}(e^{-it\LL}f)&=& -i\LL 
e^{-it\LL}f\\\label{derivativeG}
\frac{\partial}{\partial_t}\displaystyle\int_0^t e^{-i(t-s)\LL}
G(z,s)ds&=&-i\LL\displaystyle\int_0^t e^{-i(t-s)\LL}G(z,s)ds+G(z,t)
\eea
are valid in the distribution sense for 
$f \in L^2(\C^n)$, $G \in L^{q'}(I, L^{p'})$. Using these we now show the equivalence of 
the initial value problem (\ref{mainpde}), (\ref{inidata}) and the integral equations (\ref{solution}).

Note that $G(z,s,u(z,s)) 
\in  L^{q'}(I, \tilde{W}^{1,p'}(\C^n))$, whenever $u \in L^\infty(I, \tilde{W}^{1,2}(\C^n)) 
\cap L^q(I, \tilde{W}^{1,p}(\C^n))$ by Proposition \ref{Gest}. Hence if such a $u$
 satisfies (\ref{solution}) then using  (\ref{derivativef}) and 
(\ref{derivativeG}), we conclude that $u$ satisfies (\ref{mainpde}) 
and (\ref{inidata}).

On the otherhand, if $u$ satisfies  (\ref{mainpde}) and (\ref{inidata}) 
then the function $v$ given by
$$v(z,t)=u(z,t)-e^{-it\LL}f+i\displaystyle\int_0^t e^{-i(t-s)\LL}G(z,s)ds,$$
satisfies \Bea
i\partial_tv(z,t)-\LL v(z,t)=0,  \\
v(z,0)=0.
\Eea
The unique solution to this linear problem is given by 
$v(z,t)=e^{-it\LL}v(z,0) \equiv 0$ since $v(z,0)=0$. Therefore
$u$ satisfies (\ref{solution}).

Now we prove (\ref{derivativef}) and (\ref{derivativeG}).  Let $\phi \in C_c^{\infty}\left(\C^n\times I \right)$.
Since $I$ is an open interval, supp $\phi\subset A\times B$, for some compact set $A\subset \C^n$ and 
some compact interval $B\subset I$.
Clearly, 
$$\frac{\partial}{\partial_t}(e^{-it\LL}\bar{\phi})=
e^{-it\LL}\frac{\partial}{\partial_t}\bar{\phi}-e^{-it\LL}i \LL \bar{ \phi} .$$
Also  since $\phi(z,\cdot)$ has compact support in $ I$ for each $z$,  $\label{vanint} \int_I \frac{\partial}{\partial_t}(e^{-it\LL}\phi)dt=0,$
hence
\bea \label{intequl} \int_I e^{-it\LL}  \frac{\partial}{\partial_t} \overline{\phi} dt=\int_I e^{-it\LL}  i \LL \overline{\phi} dt
= i ~\, \overline{ \int_I e^{it\LL}   \LL \phi dt}.\eea  
Using this and the pairing $\langle f, \varphi\rangle = \int f \bar{\varphi}$, we see that
\begin{eqnarray*}
\int_{\C^n \times I}  e^{-it\LL}f(z)\,  \frac{\partial}{\partial_t} \overline{\phi (z,t)} \, dz dt  &=& 
\left \langle e^{-it\LL}f, \frac{\partial}{\partial_t}\phi \right \rangle
=  \left \langle f, e^{it\LL}\frac{\partial}{\partial_t}\phi \right \rangle\\
&=& \left\langle f,-i e^{it\LL} \LL \phi \right\rangle 
= \left\langle  i\LL e^{-it\LL}f,\phi \right\rangle.
\end{eqnarray*}
This proves (\ref{derivativef}) in the distribution sense.

To prove (\ref{derivativeG}), choose a sequence $\{G_m\}$ 
in $C_c^{\infty}\left(A\times B \right)$ such that 
$G_m\rightarrow G$ in $L^{q'}(B,L^{p'}(A))$. 
Note that $G_m \in L^2(A \times B)$ hence, $$ \underset{h \to 0}{\lim} \, \frac{1}{h}\left [e^{-i(t+h-s)\LL}-e^{-i(t-s)\LL}\right]G_m(z,s)
= -i\LL \,  e^{-i(t-s)\LL}G_m(z,s)$$ and $\underset{s\to t}{\lim} \, e^{-i(t-s)\LL}G_m(z,s) = G_m(z,t)$
where both the limits are taken in $L^2(\C^n)$ sense. Thus as an $L^2(\C^n)$ valued 
integral on $I$, we have
\begin{eqnarray} \label{5.4}
 \frac{\partial}{\partial_t} && \!\!\!\!\!\!\!\!\!\!\!\!\!  
\int_0^t   e^{-i(t-s)\LL}G_m(z,s)ds \nonumber \\
&=& \underset{h\to 0}{\lim} \frac{1}{h} \left(\displaystyle
\int_0^{t+h} e^{-i(t+h-s)\LL}G_m(z,s)ds-\displaystyle\int_0^t 
e^{-i(t-s)\LL}G_m(z,s)ds\right)  \nonumber  \\
&=& \underset{h\to 0}{\lim}  \frac{1}{h}\displaystyle
\int_0^{t+h} (e^{-i(t+h-s)\LL}-e^{-i(t-s)\LL})G_m(z,s)ds  \nonumber \\
&& +\underset{h\to 0}{\lim} \frac{1}{h}
\displaystyle\int_t^{t+h} e^{-i(t-s)\LL}G_m(z,s)ds  \nonumber \\
&=& -i\LL \displaystyle\int_0^t e^{-i(t-s)\LL}G_m(z,s)ds+G_m(z,t).
\end{eqnarray}
Observe that 
$\displaystyle\int_0^t e^{-i(t-s)\LL}
G_m(z,s)ds\rightarrow \displaystyle\int_0^t e^{-i(t-s)\LL}G(z,s)ds$ 
in $L^q\left(B;L^p(A)\right)$ as $m\rightarrow \infty$. This follows from estimate (\ref{stauxest}) since $B$ is a bounded interval.
Thus using (\ref{5.4}), we see that 
\Bea
\left\langle \displaystyle
\int_0^t e^{-i(t-s)\LL}G(z,s)ds,\frac{\partial}{\partial_t}\phi \right\rangle 
\!\!\!\!\!&=\underset{m\to \infty}{\lim} \left\langle 
\displaystyle\int_0^t e^{-i(t-s)\LL}G_m(z,s)ds,
\frac{\partial}{\partial_t}\phi \right\rangle\\
&=\underset{m\to \infty}{\lim} \left\langle 
-\frac{\partial}{\partial_t}\displaystyle\int_0^t 
e^{-i(t-s)\LL}G_m(z,s)ds,\phi \right\rangle\\
&=\underset{m\to \infty}{\lim} 
\left\langle i\LL\displaystyle\int_0^t e^{-i(t-s)\LL}G_m(z,s)ds
-G_m(z,t),\phi \right\rangle\\
&=\underset{m\to \infty}{\lim} \left\langle 
\displaystyle\int_0^t e^{-i(t-s)\LL}G_m(z,s)ds,- i\LL\phi 
\right\rangle - \langle G(z,t),\phi \rangle\\
&=-\left\langle \displaystyle\int_0^t 
e^{-i(t-s)\LL}G(z,s)ds,i\LL\phi \right\rangle-\langle G(z,t),\phi \rangle.
\Eea
This shows that (\ref{derivativeG}) holds in the distribution sense.
\end{proof}

\end{document}